\documentclass[12pt,a4paper]{article}
\usepackage[cp1251]{inputenc}
\usepackage[hidelinks]{hyperref}
\usepackage[english,russian]{babel}
\usepackage{amsmath}
\usepackage{amsthm}
\usepackage{MPatr_Style_English_09.2014}
\usepackage{dsfont}
\usepackage{pifont}
\usepackage{amsfonts}
\usepackage{MnSymbol}
\DeclareMathOperator{\scope}{\mathsf{scope}}
\DeclareMathOperator{\rise}{\mathsf{span}}
\DeclareMathOperator{\riseee}{\mathsf{spectrum}}

\DeclareMathOperator{\domain}{\mathsf{domain}}
\DeclareMathOperator{\pset}{\mathsf{power\hspace{-0.8pt}.\hspace{-2pt}set}}

\DeclareMathOperator{\nodes}{\mathsf{nodes}}
\DeclareMathOperator{\maxel}{\mathsf{max}}
\DeclareMathOperator{\minel}{\mathsf{min}}
\DeclareMathOperator{\height}{\mathsf{height}}

\DeclareMathOperator{\level}{\mathsf{level}}

\DeclareMathOperator{\sons}{\mathsf{sons}}

\DeclareMathOperator{\impl}{\mathsf{implant}}
\DeclareMathOperator{\expl}{\mathsf{explant}}
\DeclareMathOperator{\supp}{\mathsf{support}}
\DeclareMathOperator{\hybr}{\mathsf{hybrid}}
\DeclareMathOperator{\skeleton}{\mathsf{skeleton}}

\DeclareMathOperator{\flesh}{\mathsf{flesh}}

\DeclareMathOperator{\cofin}{\mathsf{cofin}}

\DeclareMathOperator{\shoot}{\mathsf{shoot}}

\DeclareMathOperator{\cut}{\mathsf{cut}}
\DeclareMathOperator{\loss}{\mathsf{loss}}
\DeclareMathOperator{\fhybr}{\hspace{0.4pt}\mathsf{fol\hspace{-0.7pt}.\hspace{-2pt}hybr}\hspace{0.5pt}}

\DeclareMathOperator{\length}{\mathsf{length}}
\DeclareMathOperator{\nbhds}{\mathsf{nbhds}}

\author{Mikhail Patrakeev
\footnote{Krasovskii Institute of Mathematics and Mechanics of UB RAS, 620990, 16 Sofia Kovalevskaya street, Yekaterinburg, Russia and
Ural Federal University, 620002, 19 Mira street, Yekaterinburg, Russia;
patrakeev@mail.ru}
}

\title{When the property of having a $\pi\!$-tree\\is preserved by products\footnote{2010 Mathematics Subject Classification\textup{:} Primary 54{E}99; Secondary 54{H}05. Keywords\textup{:} the Sorgenfrey line, the Baire space, Souslin scheme, Lusin scheme, open sieve, Lusin pi-base, pi-tree, foliage tree, the foliage hybrid operation, product of topological spaces.}}

\date{}
\begin{document}%\linespread{1.3}
\hyphenation{no-n-in-cre-a-s-ing tran-si-ti-ve par-ti-al}
\renewcommand{\proofname}{\textup{\textbf{Proof}}}
\renewcommand{\abstractname}{\textup{Abstract}}
\renewcommand{\refname}{\textup{References}}
\mathsurround=2pt
\maketitle

\begin{abstract}
We find sufficient conditions under which the product of spaces that have a $\pi\!$-tree also has a $\pi\!$-tree.
These conditions give new examples of spaces with a $\pi\!$-tree:
every at most countable power of the Sorgenfrey line and every at most countable power of the irrational Sorgenfrey line has a $\pi\!$-tree.
Also we show that if a space has a $\pi\!$-tree, then its product with the Baire space, with the Sorgenfrey line, and with the countable power of the Sorgenfrey line also has a $\pi\!$-tree.
\end{abstract}

\section{Introduction}
\label{section0}

We study topological spaces that have a $\pi\!$-tree; the notion of a $\pi\!$-tree was introduced in~\cite{my.paper} and is equivalent~\cite[Remark~11]{my.paper} to the notion of a Lusin $\pi\!$-base, which was introduced in~\cite{MPatr}. 
%We study topological spaces that have a $\pi\!$-tree, see Section~\ref{section1}; the notion of a $\pi\!$-tree is equivalent~\cite[Remark~11]{my.paper} to the notion of a Lusin $\pi\!$-base, which was introduced in~\cite{MPatr}. 
The Sorgenfrey line $\mathcal{R}_{\scriptscriptstyle\mathcal{S}}$ and the Baire space $\mathcal{N}$ (that is, ${}^{\omega}\omega$ with the product topology) are examples of spaces with a $\pi\!$-tree~\cite{MPatr}.
Every space that has a $\pi\!$-tree shares many good properties with the Baire space. One reason for this is expressed in Lemmas~\ref{lem.pi.and.B.f.trees.vs.S} and \ref{lem.F.vs.S.by.isomorph}, another two are the following:
If a space ${X}$ has a $\pi\!$-tree, then ${X}\hspace{-1pt}$ can be mapped onto $\mathcal{N}$ by a continuous one-to-one map~\cite{MPatr} and also ${X}\hspace{-1pt}$ can be mapped onto $\mathcal{N}$ by a continuous open map~\cite{MPatr} (hence ${X}\hspace{-1pt}$ can be mapped by a continuous open map onto an arbitrary Polish space, see~\cite{arh1} or~\cite[Exercise\,7.14]{kech}).
Every space that has a $\pi\!$-tree also has a countable $\pi\!$-base, see Lemma~\ref{lem.pi.base.pseudo.base}.

In this paper we study the following question: When does the product of spaces that have a $\pi\!$-tree also have a $\pi\!$-tree? We find several kinds of conditions (see Theorems~\ref{th.1},~\ref{th.2} and Corollary~\ref{cor.Y*product}) under which an at most countable product of spaces that have a $\pi\!$-tree also has a $\pi\!$-tree. We consider only at most countable products because an uncountable product of spaces that have a $\pi\!$-tree has an uncountable pseudocharacter, therefore it has no $\pi\!$-tree (see~\cite[statement\,5.3.b]{Juh} and
Lemmas~\ref{lem.pi.and.B.f.trees.vs.S},\,\ref{lem.pi.base.pseudo.base},\,\ref{lem.F.vs.S.by.isomorph}).

The above results give new examples of spaces that have a $\pi\!$-tree, see Section~\ref{sect.examples}.
For instance, Corollary~\ref{cor.rise.in.N.and.sorg.line} assirts that
if ${1}\leqslant{|}{A}{|}\leqslant\omega$ and for each $\alpha\in{A},$
$$
\text{either}\qquad{X}_{\alpha}\hspace{-1pt}=\hspace{1pt}\mathcal{N}
\qquad\text{or}\qquad
{X}_{\alpha}\subseteq\hspace{1pt}\mathcal{R}_{\scriptscriptstyle\mathcal{S}}
\ \text{ with }\ \mathcal{R}_{\scriptscriptstyle\mathcal{S}}
\hspace{0.7pt}{\setminus}\,{X}_{\alpha}
\ \text{ at most countable,}
$$
then the product $\prod_{\alpha\in{A}}{X}_{\alpha}$ has a $\pi\!$-tree.
In particular, the powers ${\mathcal{R}_{\scriptscriptstyle\mathcal{S}}}^{{n}}\hspace{-1pt}$ and ${\mathcal{I}_{\scriptscriptstyle\mathcal{S}}}^{{n}}\hspace{-1pt}$ ($\mathcal{I}_{\scriptscriptstyle\mathcal{S}}$ denotes the irrational Sorgenfrey line $\mathcal{R}_{\scriptscriptstyle\mathcal{S}}\setminus\mathbb{Q}$) have a $\pi\!$-tree for all natural ${n}\geqslant{1},$ and the powers
${\mathcal{R}_{\scriptscriptstyle\mathcal{S}}}^{\omega}\hspace{-1pt}$ and ${\mathcal{I}_{\scriptscriptstyle\mathcal{S}}}^{\omega}\hspace{-1pt}$ also have a $\pi\!$-tree. (Note that no finite power of the irrational Sorgenfrey line is homeomorphic to finite power of the Sorgenfrey line~\cite{Pacific}.)
Other examples of spaces with a $\pi\!$-tree can be obtained by using Corollary~\ref{cor.pi.tree.for.X*N.X*Sor.l}, which says that if a space ${X}$ has a $\pi\!$-tree, then the products ${X}\times\mathcal{N},$ ${X}\times\mathcal{R}_{\scriptscriptstyle\mathcal{S}},$ and ${X}\times{\mathcal{R}_{\scriptscriptstyle\mathcal{S}}}^{\omega}\hspace{-1pt}$ also have a $\pi\!$-tree.

\section{Notation and terminology}
\label{section1}

We use standard set-theoretic notation from \cite{jech,kun}. In particular, each ordinal is equal to the set of smaller ordinals, $\omega=$ the set of natural numbers $=$ the set of finite ordinals $=$ the first limit ordinal $=$ the first infinite cardinal, and ${n}=\{0,\ldots,{n}-1\}$ for all ${n}\in\omega.$
A \emph{space} is a topological space; we use terminology from~\cite{top.enc} when we work with spaces. Also we use the following notations:

\begin{terminology}\label{not01}%
   The symbol $\coloneq$ means ``equals by definition''\textup{;}
   the symbol ${\colon}{\longleftrightarrow}$ is used to show that an expression on the left side is an abbreviation for expression on the right side\textup{;}

   \begin{itemize}
   \item [\ding{46}]
      $\mathsurround=0pt
      {x}\subset{y}
      \quad{\colon}{\longleftrightarrow}\quad
      {x}\subseteq{y}\enskip\mathsf{and}\enskip{x}\neq{y};$
   \item [\ding{46}]
      $\mathsurround=0pt
      {A}\equiv\bigsqcup_{{\lambda}\in{\Lambda}}{B}_{\lambda}
      \quad{\colon}{\longleftrightarrow}\quad
      {A}=\bigcup_{{\lambda}\in{\Lambda}}{B}_{\lambda}\enskip\mathsf{ and }\enskip
      \forall{\lambda},{\lambda}'\in{\Lambda}
      \;[{\lambda}\neq{\lambda}'\to{B}_{\lambda}\cap{B}_{{\lambda}'}=\varnothing];$
   \item [\ding{46}]
      $\mathsurround=0pt
      [{A}]^{\kappa}\coloneq
      \big\{{B}\subseteq{A}:{|}{B}{|}=\kappa\big\},\quad
      [{A}]^{{<}\kappa}\coloneq
      \big\{{B}\subseteq{A}:{|}{B}{|}<\kappa\big\}
      \quad$\textup{(}here $\kappa$ is a cardinal\textup{);}
   \item [\ding{46}]
      $\mathsurround=0pt
      \gamma\,$ has the \textup{\textsf{FIP}}$\quad{\colon}{\longleftrightarrow}\quad
      \forall\delta\in[\gamma]^{{<}\omega}{\setminus}\{\varnothing\}
      \ \big[\bigcap\delta\neq\varnothing\big]
      \quad$\textup{(\textsf{FIP}} means finite intersection property\textup{);}
   \item [\ding{46}]
      $\mathsurround=0pt
      \cofin{A}\coloneq\big\{{A}\setminus{F}:{F}\in[{A}]^{{<}\omega}\big\};$
   \item [\ding{46}]
      $\mathsurround=0pt
      \nbhds({p},{X})\coloneq$
      the set of (not necessarily open) neighbourhoods of point ${p}$ in space~${X};$
   \item [\ding{46}]
      $\mathsurround=0pt
      {f}{\upharpoonright}{A}\coloneq$ the restriction of function ${f}$ to ${A};$
   \item [\ding{46}]
      $\mathsurround=0pt
      \gamma\gg\delta
      \quad{\colon}{\longleftrightarrow}\quad
      \gamma\,$ $\pi\!$-\emph{refines} $\,\delta
      \quad{\colon}{\longleftrightarrow}\quad
      \forall\hspace{-1pt}{D}\,{\in}\:\delta\hspace{0.7pt}{\setminus}\{\varnothing\}
      \ \,\exists{G}\,{\in}\,\gamma\hspace{0.3pt}{\setminus}\{\varnothing\}
      \ \,[\,{G}\subseteq{D}\,].$
   \end{itemize}
\end{terminology}

When we work with (transfinite) sequences, we use the following notations:

\begin{terminology}\label{not02}%
   Suppose ${n}\in\omega$ and ${s},{t}$ are sequences; that is, ${s}$ and ${t}$ are functions whose domain is an ordinal.

   \begin{itemize}
   \item [\ding{46}]
      $\mathsurround=0pt
      \length{s}\coloneq$ the domain of ${s};$
   \item [\ding{46}]
      note that
      $\;{s}\subseteq{t}
      \enskip\;\:\mathsf{iff}\enskip\:
      \length{s}\leqslant\length{t}\enskip\text{and}\enskip
      {s}={t}{\upharpoonright}\length{s};$
   \item [\ding{46}]
      $\mathsurround=0pt
      \langle{r}_0,\ldots,{r}_{{n}-1}\rangle\coloneq$
      the sequence ${r}$ such that  $\length{r}={n}$ and ${r}(i)={r}_{i}$ for all ${i}\in{n};$
   \item [\ding{46}]
      $\mathsurround=0pt
      \langle\rangle\coloneq$ the sequence of length {0};
   \item [\ding{46}]
      $\mathsurround=0pt
      \langle{r}_0,\ldots,{r}_{{n}-1}\rangle\hat{\:}\langle{s}_0,\ldots,{s}_{{m}-1}\rangle
      \coloneq\langle{r}_0,\ldots,{r}_{{n}-1},{s}_0,\ldots,{s}_{{m}-1}\rangle;$
   \item [\ding{46}]
      $\mathsurround=0pt
      {}^{B}\!{A}\coloneq$ the set of functions from ${B}$ to ${A};$ in particular, ${}^{0}\hspace{-1pt}{A}=\big\{\langle\rangle\big\};$
   \item [\ding{46}]
      $\mathsurround=0pt
      {}^{{<}\alpha}\hspace{-1pt}{A}
      \coloneq\bigcup_{\beta\in\alpha}{}^\beta\hspace{-1.5pt}{A}
      \quad$\textup{(}here $\alpha$ is an ordinal\textup{)}.
   \end{itemize}
\end{terminology}

Also we work with partial orders and then we use the following terminology:
\begin{terminology}
   Suppose
   $\mathcal{P}=({Q},{\vartriangleleft})$ is a strict partial order; that is, ${\vartriangleleft}$ is irreflexive and transitive on~${Q}.$ Let ${x},{y}\in{Q}$ and ${A}\subseteq {Q}.$

   \begin{itemize}
   \item [\ding{46}]
      $\mathsurround=0pt
      \nodes\mathcal{P}=\nodes({Q},{\vartriangleleft})
      \coloneq{Q};$
   \item [\ding{46}]
      $\mathsurround=0pt
      {x}<_\mathcal{P} {y}
      \quad{\colon}{\longleftrightarrow}\quad
      {x}\vartriangleleft{y};$
   \item [\ding{46}]
      $\mathsurround=0pt
      {x}\leqslant_\mathcal{P} {y}
      \quad{\colon}{\longleftrightarrow}\quad
      {x}<_\mathcal{P}{y}\enskip\mathsf{or}\enskip{x}={y};$
   \item [\ding{46}]
      $\mathsurround=0pt
      {x}{\upspoon}_{\hspace{-1.5pt}\mathcal{P}}\coloneq\{{v}\in
      \nodes\mathcal{P}:{v}<_\mathcal{P} {x}\},\quad
      {x}{\downspoon}_{\hspace{-0.5pt}\mathcal{P}}\coloneq
      \{{v}\in\nodes\mathcal{P}:{v}>_\mathcal{P} {x}\};$
   \item [\ding{46}]
      $\mathsurround=0pt
      {x}{\upfilledspoon}_{\hspace{-1.5pt}\mathcal{P}}\coloneq
      \{{v}\in\nodes\mathcal{P}:{v}\leqslant_\mathcal{P} {x}\},\quad {x}{\downfilledspoon}_{\hspace{-0.5pt}\mathcal{P}}\coloneq
      \{{v}\in\nodes\mathcal{P}:{v}\geqslant_\mathcal{P} {x}\};$
   \item [\ding{46}]
      $\mathsurround=0pt
      {A}{\upfootline}_{\hspace{-1.5pt}\mathcal{P}}\coloneq
      \bigcup\{{v}{\upfilledspoon}_{\hspace{-1.5pt}\mathcal{P}}:{v}\in{A}\},\quad
      {A}{\downfootline}_{\hspace{-0.5pt}\mathcal{P}}\coloneq
      \bigcup\{{v}{\downfilledspoon}_{\hspace{-0.5pt}\mathcal{P}}:{v}\in{A}\};$
   \item [\ding{46}]
      $\mathsurround=0pt
      \sons_{\mathcal{P}}({x})\coloneq
      \{{s}\in\nodes\mathcal{P}:{x}<_\mathcal{P}{s}\enskip\mathsf{and}\enskip
      {x}{\downspoon}_{\hspace{-0.5pt}\mathcal{P}}\cap
      {s}{\upspoon}_{\hspace{-1.5pt}\mathcal{P}}=\varnothing\};$
   \item [\ding{46}]
      $\mathsurround=0pt
      {A}\,$ is a \emph{chain} in $\mathcal{P}
      \quad{\colon}{\longleftrightarrow}\quad
      \forall{v},{w}\in {A}
      \;[
      {v}\leqslant_\mathcal{P} {w}\enskip\mathsf{or}\enskip{v}>_\mathcal{P} {w}
      ];$
   \item [\ding{46}]
      $\mathsurround=0pt
      \mathcal{P}\,$ has \emph{bounded chains}$\quad{\colon}{\longleftrightarrow}\quad$

      for each nonempty chain ${C}$ in $\mathcal{P}$ there is ${v}\in\nodes\mathcal{P}$ such that
      ${C}\subseteq {v}{\upfilledspoon}_{\hspace{-1.5pt}\mathcal{P}};$
   \item [\ding{46}]
      $\mathsurround=0pt
      \maxel{\mathcal{P}}\coloneq
      \{{m}\in\nodes\mathcal{P}:
      {m}{\downspoon}_{\hspace{-0.5pt}\mathcal{P}}=\varnothing\},\quad
      \minel{\mathcal{P}}\coloneq
      \{{m}\in\nodes\mathcal{P}:
      {m}{\upspoon}_{\hspace{-1.5pt}\mathcal{P}}=\varnothing\};$
   \item [\ding{46}]
      $\mathsurround=0pt
      {0}_{\mathcal{P}}\coloneq$ the node such that
      $({0}_{\mathcal{P}}){\downfilledspoon}_{\hspace{-0.5pt}\mathcal{P}}=\nodes\mathcal{P}\quad$(here $\mathcal{P}$ is a partial order that has such node).
   \end{itemize}
\end{terminology}

When a partial order is a (set-theoretic) tree, we use the following terminology:

\begin{terminology}\label{not.trees}
   Suppose
   $\mathcal{T}\hspace{-1pt}$ is a tree;  that is, $\mathcal{T}\hspace{-1pt}$ is a strict partial order
   such that for each ${x}\in\nodes\mathcal{T}\hspace{-1pt},$ the set ${x}{\upspoon}_{\hspace{-1.5pt}\mathcal{T}}$ is well-ordered by ${<}_{\mathcal{T}}.$  Let ${x}\in\nodes\mathcal{T}\hspace{-1pt},$ let $\alpha$ be  an ordinal, and let $\kappa$ be a cardinal.

   \begin{itemize}
   \item [\ding{46}]
      $\mathsurround=0pt
      \height_{\mathcal{T}}({x})\coloneq$
      the ordinal isomorphic to $( {x}{\upspoon}_{\hspace{-1.5pt}\mathcal{T}}, <_{\mathcal{T}});$
   \item [\ding{46}]
      $\mathsurround=0pt
      \level_{\mathcal{T}}(\alpha)\coloneq\big\{{v}\in\nodes\mathcal{T}:
      \height_{\mathcal{T}}({v})=\alpha\big\};$
   \item [\ding{46}]
      $\mathsurround=0pt
      \height\mathcal{T}\coloneq$
      the minimal ordinal $\beta$ such that $\level_{\mathcal{T}}(\beta)=\varnothing;$
   \item [\ding{46}]
      $\mathsurround=0pt
      {B}\,$ is a \emph{branch} in $\mathcal{T}
      \quad{\colon}{\longleftrightarrow}\quad
      {B}$ is a ${\subseteq}\!$-maximal chain in $\mathcal{T};$
   \item [\ding{46}]
      $\mathsurround=0pt
      \mathcal{T}\,$ is $\kappa\!$-\emph{branching} $\quad{\colon}{\longleftrightarrow}\quad
      \forall{v}\in\nodes\mathcal{T}\setminus\maxel\mathcal{T}
      \ \big[
      {|}\sons_{\mathcal{T}}({v}){|}=\kappa
      \big].$
   \end{itemize}
\end{terminology}

Finally, we work with foliage trees, which where introduced in \cite{my.paper}.
Recall that a \emph{foliage tree} is a pair $\mathbf{F}=(\mathcal{T}\hspace{-1pt},{l}) $ such that $\mathcal{T}\hspace{-1pt}$ is a tree and ${l}$~is a function with $\domain{l}=\nodes\mathcal{T}.$
For each ${x}\in\nodes\mathcal{T}\hspace{-1pt},$ the ${l}({x})$ is called the \emph{leaf} of $\mathbf{F}\hspace{-1pt}$ at node ${X}\hspace{-1pt}$ and is denoted by $\mathbf{F}_{\hspace{-1.2pt}{x}};$ the tree $\mathcal{T}\hspace{-1pt}$ is called the \emph{skeleton} of~$\mathbf{F}\hspace{-1pt}$ and is denoted by $\skeleton\mathbf{F}.$ We adopt the following convention: If $\mathbf{F}\hspace{-1pt}$ is a foliage tree and $\bullet$ is a notation that can be applied to a tree, then $\bullet(\mathbf{F})$ is an abbreviation for $\bullet(\skeleton\mathbf{F});$ for example,
${x}<_\mathbf{F}{y}$ stands for ${x}<_{\skeleton\mathbf{F}}{y}.$ Also we use the following terminology:

\begin{terminology}\label{not.fol.trees}
   Suppose $\mathbf{F}\hspace{-1pt}$ is a foliage tree, ${v}\in\nodes\mathbf{F},$ ${A}\subseteq\nodes\mathbf{F},$ ${X}\hspace{-1pt}$ is a space, $\alpha$ is an ordinal, and $\kappa$ is a cardinal.

   \begin{itemize}
   \item [\ding{46}]
      $\mathsurround=0pt
      \flesh\mathbf{F}\coloneq
      \bigcup\{\mathbf{F}_{\hspace{-1.2pt}{x}}:{x}\in\nodes\mathbf{F}\};$
   \item [\ding{46}]
      $\mathsurround=0pt
      \flesh_\mathbf{F}({A})\coloneq
      \bigcup\{\mathbf{F}_{\hspace{-1.2pt}{x}}:{x}\in{A}\};$
   \item [\ding{46}]
      $\mathsurround=0pt
      \shoot_\mathbf{F}({v})\coloneq
      \big\{\flesh_\mathbf{F}({C}):
      {C}$ is a cofinite subset of $\sons_{\mathbf{F}}({v})\big\};$
   \item [\ding{46}]
      $\mathsurround=0pt
      \scope_\mathbf{F}({a})\coloneq
      \{{x}\in\nodes\mathbf{F}:\mathbf{F}_{\hspace{-1.2pt}{x}}\ni{a}\};$
   \item [\ding{46}]
      $\mathsurround=0pt
      \mathbf{F}\,$ has \emph{nonempty leaves}$
      \quad{\colon}{\longleftrightarrow}\quad
      \forall{x}\in\nodes\mathbf{F}
      \;[\mathbf{F}_{\hspace{-1.2pt}{x}}\neq\varnothing
      ];$
   \item [\ding{46}]
      $\mathsurround=0pt
      \mathbf{F}\,$ is \emph{nonincreasing}$
      \quad{\colon}{\longleftrightarrow}\quad
      \forall{x},{y}\in\nodes\mathbf{F}
      \;[
      {y}\geqslant_{\mathbf{F}} {x}\rightarrow\mathbf{F}_{\hspace{-1.2pt}{y}}\subseteq\mathbf{F}_{\hspace{-1.2pt}{x}}
      ];$
   \item [\ding{46}]
      $\mathsurround=0pt
      \mathbf{F}\,$ has \emph{strict branches}$
      \quad{\colon}{\longleftrightarrow}\quad$

      $\mathsurround=0pt
      \nodes\mathbf{F}\neq\varnothing\,$ and for each branch ${B}$ in $\mathbf{F},$ the
      $\bigcap_{{x}\in{B}}\mathbf{F}_{\hspace{-1.2pt}{x}}$ is a singleton;
   \item [\ding{46}]
      $\mathsurround=0pt
      \mathbf{F}\,$ is \emph{locally strict}$
      \quad{\colon}{\longleftrightarrow}\quad
      \forall{x}\in\nodes\mathbf{F}\,{\setminus}\maxel\mathbf{F}
      \ [\mathbf{F}_{\hspace{-1.2pt}{x}}\equiv
      \bigsqcup_{{s}\in\sons_{\mathbf{F}}
      ({x})}\mathbf{F}_{\hspace{-1.2pt}{s}}];$
   \item [\ding{46}]
      $\mathsurround=0pt
      \mathbf{F}\,$ is \emph{open} in  ${X}
      \quad{\colon}{\longleftrightarrow}\quad
      \forall{z}\in\nodes\mathbf{F}
      \ [\mathbf{F}_{\hspace{-1.2pt}{z}}$ is an open subset of ${X}
      ];$
   \item [\ding{46}]
      $\mathsurround=0pt
      \mathbf{F}\,$ is a \emph{foliage} $\alpha,\kappa\!$-\emph{tree}$
      \quad{\colon}{\longleftrightarrow}\quad
      \skeleton\mathbf{F}\hspace{-1pt}$
      is isomorphic to the tree $({}^{{<}\alpha}\kappa,{\subset});$
   \item [\ding{46}]
      $\mathsurround=0pt
      \mathbf{F}\,$ is a \emph{Baire foliage tree} on  ${X}
      \quad{\colon}{\longleftrightarrow}\quad
      \mathbf{F}\hspace{-1pt}$ is an open in ${X}$ locally strict foliage $ \omega,\omega\!$-tree with strict branches
      and such that $\mathbf{F}_{\hspace{-1.2pt}{0}_{\mathbf{F}}}={X};$
   \item [\ding{46}]
      $\mathsurround=0pt
      \mathbf{F}\,$ \emph{grows into} ${X}
      \quad{\colon}{\longleftrightarrow}\quad
      \forall\hspace{-1pt}{p}\hspace{1pt}{\in}{X}\ \forall{U}{\in}\nbhds({p},{X})\ \exists{z}{\in}\scope_\mathbf{F}({p})
      \ \big[\shoot_\mathbf{F}({z})\gg\{{U}\}\big];$
   \item [\ding{46}]
      $\mathsurround=0pt
      \mathbf{F}\,$ is a $\pi\!$-\emph{tree} on  ${X}
      \quad{\colon}{\longleftrightarrow}\quad
      \mathbf{F}\hspace{-1pt}$ is a Baire foliage tree on ${X}\hspace{-1pt}$ and $\mathbf{F}\hspace{-1pt}$ grows into ${X};$
   \item [\ding{46}]
      $\mathsurround=0pt
      \mathbf{S}$ $\coloneq$ the \emph{standard foliage tree} of ${}^{\omega}\omega$ $\coloneq$ the foliage tree such that
      \begin{itemize}
      \item[\ding{226}]
         $\mathsurround=0pt
         \skeleton\mathbf{S}\coloneq({}^{{<}\omega}\omega,\subset)\enskip$and
      \item[\ding{226}]
         $\mathsurround=0pt
         \mathbf{S}_{{x}}\coloneq\{{p}\in{}^{\omega}\omega:
         {x}\subseteq{p}\}\enskip$for every ${x}\in{}^{{<}\omega}\omega;$
      \end{itemize}
   \item [\ding{46}]
      $\mathsurround=0pt
      \mathcal{N}$ $\coloneq$ the \emph{Baire space} $\coloneq$ the space $({}^{\omega}\omega,\tau_{\scriptscriptstyle\mathcal{N}}),$ where $\tau_{\scriptscriptstyle\mathcal{N}}$ is the Tychonoff product topology with $\omega$ carrying %endowed with
      the discrete topology.
   \end{itemize}
\end{terminology}

\begin{lem}[{Lemma~13 in \cite{my.paper}}]\label{lem.pi.and.B.f.trees.vs.S}\mbox{ }

   \begin{itemize}
   \item[\textup{(a)}]
      $\mathsurround=0pt
      \hspace{-0.7pt}\{\mathbf{S}_{{x}}:{x}\in{}^{{<}\omega}\omega\}\,$
      is a base for $\mathcal{N}.$
   \item[\textup{(b)}]
      $\mathsurround=0pt
      \mathbf{S}\,$ is a $\pi\!$-tree on $\mathcal{N}.$
   \item[\textup{(c)}]
      $\mathsurround=0pt
      \mathbf{S}\,$ is a Baire foliage tree on a space $({}^{\omega}\omega,\tau)
      \quad\textup{\textsf{iff}}\quad
      \tau\supseteq\tau_{\scriptscriptstyle\mathcal{N}}.$
   \end{itemize}
\end{lem}

\begin{lem}\label{lem.pi.base.pseudo.base}
   %If a space ${X}\hspace{-1pt}$ has a $\pi\!$-tree $\mathbf{F},$ then
   If $\,\mathbf{F}\hspace{-1pt}$ is a $\pi\!$-tree on a space ${X},$ then

   \begin{itemize}
   \item[\ding{226}]
      $\mathsurround=0pt
      \hspace{-0.7pt}\{\mathbf{F}_{\hspace{-1.2pt}{v}}:{v}\in\nodes\mathbf{F}\}\,$ is a countable $\pi\!$-base for ${X},$
   \item[\ding{226}]
      each $\mathbf{F}_{\hspace{-1.2pt}{v}}$ is closed-and-open in ${X},$ and
      %   \item[\textup{(b)}]
      %      $\mathsurround=0pt
      %      \mathbf{F}_{\hspace{-1.2pt}{v}}\,$ is closed-and-open in ${X}\hspace{-1pt}$ for all ${v}\in\nodes\mathbf{F},$ and
   \item[\ding{226}]
      $\mathsurround=0pt
      \bigcap\{\mathbf{F}_{\hspace{-1.2pt}{v}}:\mathbf{F}_{\hspace{-1.2pt}{v}}\ni{p}\}=\{{p}\}\quad$ for all ${p}\in\!{X}.$
      \hfill$\qed$
   \end{itemize}
\end{lem}

\section{New notions: isomorphism and spectrum}

The notion of isomorphism between foliage trees allows to simplify proofs (see the proof of Theorem~\ref{th.1}) in the following way: When we have a $\pi\!$-tree $\mathbf{F}\hspace{-1pt}$ on a space ${X},$ we may (by using (c) of Lemma~\ref{lem.F.vs.S.by.isomorph} and (c) of Lemma~\ref{lem.pi.and.B.f.trees.vs.S}) assume ``without loss of generality'' that $\mathbf{F}=\mathbf{S}$ and ${X}=({}^{\omega}\omega,\tau)$ with $\tau\supseteq\tau_{\scriptscriptstyle\mathcal{N}}.$

%The notion of isomorphism between foliage trees gives a strengthening of Lemma~\ref{lem.pi.and.B.f.trees.vs.S.2}, see Lemma~\ref{lem.F.vs.S.by.isomorph}. This lemma, in combination with (c) of Lemma~\ref{lem.pi.and.B.f.trees.vs.S}, allows to simplify proofs (see the proof of Theorem~\ref{th.1}) in the following way: when we have a $\pi\!$-tree $\mathbf{F}\hspace{-1pt}$ on a space ${X},$ we may assume without loss of generality that $\mathbf{F}=\mathbf{S}$ and ${X}=({}^{\omega}\omega,\tau)$ with $\tau\supseteq\tau_{\scriptscriptstyle\mathcal{N}}.$

\begin{deff}\label{def.isomorph}
   An \textbf{isomorphism} between foliage trees $\mathbf{F}\hspace{-1pt}$ and $\mathbf{G}$ is a pair
   $(\varphi,\psi)$ such that

   \begin{itemize}
   \item[\ding{226}]
      $\mathsurround=0pt
      \varphi\,$ is an order isomorphism from $\skeleton\mathbf{F}\hspace{-1pt}$ onto $\skeleton\mathbf{G},$
   \item[\ding{226}]
      $\mathsurround=0pt
      \psi\,$ is a bijection from $\flesh\mathbf{F}\hspace{-1pt}$ onto $\flesh\mathbf{G},$ and
   \item[\ding{226}]
      $\mathsurround=0pt
      \psi[\mathbf{F}_{\hspace{-1.2pt}{x}}]=\mathbf{G}_{\varphi({x})}\quad$ for all ${x}\in\nodes\mathbf{F}.$
   \end{itemize}
\end{deff}

\begin{lem}\label{lem.F.vs.S.by.isomorph}
   Suppose that $\hspace{1pt}\mathbf{F}\hspace{-1pt}$ is a foliage tree and ${X}\hspace{-1pt}$ is a space.

   \begin{itemize}
   \item[\textup{(a)}]
      $\mathsurround=0pt
      \mathbf{F}\,$ is a locally strict foliage $\omega,\omega\!$-tree with strict branches
      \quad\textup{\textsf{iff}}

      $\mathsurround=0pt
      \mathbf{F}\,$ is isomorphic to $\mathbf{S}.$
   \item[\textup{(b)}]
      $\mathsurround=0pt
      \mathbf{F}\,$ is a Baire foliage tree on ${X}\hspace{-1pt}$
      \quad\textup{\textsf{iff}}

      there exist an isomorphism $(\varphi,\psi)$ between $\mathbf{F}\hspace{-1pt}$ and $\mathbf{S}$ and a topology $\tau$ on ${}^{\omega}\omega$ such that
      \begin{itemize}
      \item[\ding{226}]
         $\mathsurround=0pt
         \psi\,$ is a homeomorphism from ${X}\hspace{-1pt}$ onto $({}^{\omega}\omega,\tau)$ and
      \item[\ding{226}]
         $\mathsurround=0pt
         \mathbf{S}\,$ is a Baire foliage tree on $({}^{\omega}\omega,\tau).$
      \end{itemize}
   \item[\textup{(c)}]
      $\mathsurround=0pt
      \mathbf{F}\,$ is a $\pi\!$-tree on ${X}\hspace{-1pt}$
      \quad\textup{\textsf{iff}}

      there exist an isomorphism $(\varphi,\psi)$ between $\mathbf{F}\hspace{-1pt}$ and $\mathbf{S}$ and a topology $\tau$ on ${}^{\omega}\omega$ such that
      \begin{itemize}
      \item[\ding{226}]
         $\mathsurround=0pt
         \psi\,$ is a homeomorphism from ${X}\hspace{-1pt}$ onto $({}^{\omega}\omega,\tau)$ and
      \item[\ding{226}]
         $\mathsurround=0pt
         \mathbf{S}\,$ is a $\pi\!$-tree on $({}^{\omega}\omega,\tau).$
      \end{itemize}
   \end{itemize}
\end{lem}

\begin{proof}
   (a) Suppose that $\mathbf{F}\hspace{-1pt}$ is a locally strict foliage $\omega,\omega\!$-tree with strict branches. Let $\varphi$ be an order isomorphism from $\skeleton\mathbf{F}\hspace{-1pt}$ onto the tree $({}^{{<}\omega}\omega,\subset)=\skeleton\mathbf{S}.$ For each ${p}\in{}^{\omega}\omega,$ the set $\{{x}\in{}^{{<}\omega}\omega:{x}\subseteq{p}\}$ is a branch in $\mathbf{S},$ so since $\mathbf{F}$ has strict branches it follows that there is a point $\chi({p})$ in $\flesh\mathbf{F}\hspace{-1pt}$ such that
   $$
   \textstyle\big\{\chi({p})\big\}=\bigcap\{\mathbf{F}_{\hspace{-1.2pt}\varphi^{{-}{1}}({x})}:
   {x}\in{}^{{<}\omega}\omega\enskip\textsf{and}\enskip{x}\subseteq{p}\}.
   $$
   Then it is not hard to prove that the function $\chi\colon{}^{\omega}\omega\to\flesh\mathbf{F}\hspace{-1pt}$ is a bijection and $(\varphi,\chi^{{-}{1}})$ is an isomorphism between $\mathbf{F}\hspace{-1pt}$ and $\mathbf{S}.$ The ${\leftarrow}$ direction follows from (b) of Lemma~\ref{lem.pi.and.B.f.trees.vs.S}.

   (b) Suppose that $\mathbf{F}\hspace{-1pt}$ is a Baire foliage tree on ${X}.$
   Let $(\varphi,\psi)$ be an isomorphism between $\mathbf{F}\hspace{-1pt}$ and $\mathbf{S},$ which exists by (a).
   Then $\psi$ is a bijection from ${X}\hspace{-1pt}$ onto ${}^{\omega}\omega.$
   Put
   $$
   \tau\coloneq\big\{\psi[{U}]:{U}\text{ is an open subset of }{X}\big\};
   $$
   clearly, $\tau$ is a topology on ${}^{\omega}\omega$ and $\psi$ is a homeomorphism from ${X}\hspace{-1pt}$ onto $({}^{\omega}\omega,\tau).$ It follows that $\mathbf{S}$ is a Baire foliage tree on $({}^{\omega}\omega,\tau)$ because $\mathbf{F}\hspace{-1pt}$ is a Baire foliage tree on ${X}.$ The ${\leftarrow}$ direction is similar. Part (c) can be proved by the same argument.
\end{proof}

%$\qed$

\begin{sle}\label{cor.scope.in.Baire.f.tree}
   Suppose that $\hspace{1pt}\mathbf{F}\hspace{-1pt}$ is a Baire foliage tree on a space ${X}\hspace{-1pt}$ and ${p}\in{X}.$

   \begin{itemize}
   \item [\textup{(a)}]
      $\mathsurround=0pt
      \mathbf{F}\,$ is nonincreasing, $\flesh\mathbf{F}=\mathbf{F}_{\hspace{-1.2pt}{0}_{\mathbf{F}}},$ and $\,\height\mathbf{F}=\omega;$
   \item [\textup{(b)}]
      $\mathsurround=0pt
      \mathbf{F}_{\hspace{-1.2pt}{v}}\,$ is closed-and-open in ${X}\hspace{-1pt}$ and
      $\:{|}\mathbf{F}_{\hspace{-1.2pt}{v}}{|}={2}^{\omega}\enskip$ for all ${v}\in\nodes\mathbf{F};$
   \item [\textup{(c)}]
      $\mathsurround=0pt
      \scope_{\mathbf{F}}({p})\,$ is a branch in $\mathbf{F};$
      %   \item [\textup{(f)}]
      %      $\mathsurround=0pt
      %      {\displaystyle\bigcap}\big\{\mathbf{F}_{\hspace{-1.2pt}{v}}:{v}\in\scope_{\mathbf{F}}({p})\big\}=\{{p}\};$
   \item [\textup{(d)}]
      $\mathsurround=0pt
      \forall\hspace{-1pt}{n}{\in}\hspace{1pt}\omega\ \exists!{v}{\in}\hspace{-1pt}\scope_{\mathbf{F}}({p})
      \ \big[\height_{\mathbf{F}}({v})={n}\big].$
      %   \item [\textup{(e)}]
      %      $\mathsurround=0pt
      %      \{\mathbf{F}_{\hspace{-1.2pt}{v}}:{v}\in\nodes\mathbf{F}\}\,$
      %      is a countable pseudo-base for ${X};$
      %   \item [\textup{(f)}]
      %      if $\,\mathbf{F}\hspace{-1pt}$ is a $\pi\!$-tree on ${X},$
      %
      %      then $\{\mathbf{F}_{\hspace{-1.2pt}{v}}:{v}\in\nodes\mathbf{F}\}\,$
      %      is a countable $\pi\!$-base for ${X}.$
   \end{itemize}
\end{sle}

\begin{proof}
  This corollary is a consequence of (b) of Lemma~\ref{lem.F.vs.S.by.isomorph} and (c) of Lemma~\ref{lem.pi.and.B.f.trees.vs.S}.
\end{proof}

Now we introduce terminology that we need to formulate Theorems~\ref{th.1} and~\ref{th.2}.
%The following terminology is needed to formulate Theorems~\ref{th.1} and~\ref{th.2}.

\begin{deff}\label{not.flesh}
   Suppose $\mathbf{F}\hspace{-1pt}$ is a foliage tree and ${X}\hspace{-1pt}$ is a space.

   \begin{itemize}
   \item [\ding{46}]
      $\mathsurround=0pt
      \rise_{\mathbf{F}}({p},{U})\coloneq
      \big\{\height_{\mathbf{F}}({v}):{v}\in\scope_\mathbf{F}({p})\enskip\textsf{and}\enskip\shoot_{\mathbf{F}}({v})\gg\{{U}\}\big\};$
   \item [\ding{46}]
      $\mathsurround=0pt
      \riseee_{\mathbf{F}}({X})\coloneq
      \big\{\rise_{\mathbf{F}}({p},{U}):{p}\in{X}\enskip\textsf{and}\enskip{U}\in\nbhds({p},{X})\big\}.$
   \end{itemize}
\end{deff}

\begin{pri}\label{exmpl.rise.in.S}
   $\rise_{\mathbf{S}}({p},\mathbf{S}_{{p}{\upharpoonright}\hspace{0.3pt}{n}})=\omega\setminus{n}\enskip$ for all ${p}\in{}^{\omega}\omega$ and ${n}\in\omega.$

\end{pri}

\begin{lem}\label{rem.rise.vs.grow.into}
   Suppose that $\mathbf{F}\hspace{-1pt}$ is a foliage tree and ${X}\hspace{-1pt}$ is a space.
   \begin{itemize}

   \item[\textup{(a)}]
      $\mathsurround=0pt
      \vphantom{\bar{\big\langle\rangle}}
      \mathbf{F}\,$ grows into ${X}
      \quad\textup{\textsf{iff}}\quad
      \,\varnothing\nin\riseee_{\mathbf{F}}({X}).$
   \item[\textup{(b)}]
      If $\,\mathbf{F}\hspace{-1pt}$ is a $\pi\!$-tree on ${X}\hspace{-1pt}$ and ${p}\in{X},$
      then%\textup{:}
      \begin{itemize}

      \item[\textup{(b1)}]
         the family $\big\{\rise_{\mathbf{F}}({p},{U}):{U}\in\nbhds({p},{X})\big\}$ has the \textup{\textsf{FIP}},
      \item[\textup{(b2)}]
         $\mathsurround=0pt
         {\displaystyle\bigcap}\big\{\rise_{\mathbf{F}}({p},{U}):{U}\in\nbhds({p},{X})\big\}=\varnothing,
         \vphantom{\bar{\big\langle\rangle}}$
         and
      \item[\textup{(b3)}]
         $\mathsurround=0pt
         \rise_{\mathbf{F}}({p},{U})\in[\omega]^{\omega}\enskip
         \vphantom{\bar{\big\langle\rangle}}$
         for all ${U}\in\nbhds({p},{X}).$
      \end{itemize}
   \end{itemize}
\end{lem}

\begin{proof}
   Part (a) is trivial.

   (b1) We must show that
   $$
   \text{if }\quad
   \varepsilon\in[\nbhds({p},{X})]^{{<}\omega}\setminus\{\varnothing\},
   \quad\text{ then }\quad\bigcap_{{U}\in\hspace{1pt}\varepsilon}\rise_{\mathbf{F}}({p},{U})\neq\varnothing.
   $$
   For each ${U}\in\varepsilon,$ we have
   $\rise_{\mathbf{F}}({p},{U})\supseteq\rise_{\mathbf{F}}({p},\bigcap\varepsilon)$
   because ${U}\supseteq\bigcap\varepsilon\neq\varnothing.$
   Therefore
   $$
   \bigcap_{{U}\in\hspace{1pt}\varepsilon}\rise_{\mathbf{F}}({p},{U})\supseteq
   \rise_{\mathbf{F}}({p},{\textstyle\bigcap\varepsilon})
   $$
   and it follows from (a) that $\rise_{\mathbf{F}}({p},\bigcap\varepsilon)\neq\varnothing$ since $\bigcap\varepsilon\in\nbhds({p},{X}).$

   (b2) By (c) of Lemma~\ref{lem.F.vs.S.by.isomorph}, there exist an isomorphism $(\varphi,\psi)$ between $\mathbf{F}\hspace{-1pt}$ and $\mathbf{S}$ and a topology $\tau$ on ${}^{\omega}\omega$ such that $\psi$ is a homeomorphism from ${X}\hspace{-1pt}$ onto $({}^{\omega}\omega,\tau)$ and $\mathbf{S}$ is a $\pi\!$-tree on $({}^{\omega}\omega,\tau).$ Put ${q}\coloneq\psi({p}).$
   For each ${U}\subseteq{X},$ we have $\rise_{\mathbf{F}}({p},{U})=\rise_{\mathbf{S}}\big({q},\psi[{U}]\big)$ and
   $$
   {U}\in\nbhds({p},{X})
   \enskip\leftrightarrow\enskip%\quad\textsf{iff}\quad
   \psi[{U}]\in\nbhds\big({q},({}^{\omega}\omega,\tau)\big).
   $$
   Then it is enough to show that
   $$
   \text{the set }\quad
   {M}_{{q}}\:\coloneq\ \bigcap\Big\{\rise_{\mathbf{S}}({q},{V}):
   {V}\in\nbhds\big({q},({}^{\omega}\omega,\tau)\big)\Big\}
   \quad\text{ is empty.}
   $$
   It follows from Lemma~\ref{lem.pi.and.B.f.trees.vs.S} that $\mathbf{S}_{{q}{\upharpoonright}\hspace{0.3pt}{n}}\in\nbhds\big({q},({}^{\omega}\omega,\tau)\big)$
   for all ${n}\in\omega,$
   so using Example~\ref{exmpl.rise.in.S} we have
   $$
   {M}_{{q}}\;\subseteq\;
   \bigcap\big\{\rise_{\mathbf{S}}({q},\mathbf{S}_{{q}{\upharpoonright}\hspace{0.3pt}{n}}):
   {n}\in\omega\big\}\;=\;
   {\textstyle\bigcap\{\omega\setminus{n}:{n}\in\omega\}\;=\;\varnothing}.
   $$

   (b3) It follows from (b1)--(b2) that the set $\rise_{\mathbf{F}}({p},{U})$ is infinite for all ${U}\in\nbhds({p},{X}),$ and $\rise_{\mathbf{F}}({p},{U})\subseteq\omega$ because $\height\mathbf{F}=\omega$ by (a) of Corollary~\ref{cor.scope.in.Baire.f.tree}.
\end{proof}

\section{The first theorem}

\begin{teo}\label{th.1}
   Suppose that $\mathbf{H}(\lambda)$ is a $\pi\!$-tree on a space ${X}_{\lambda}$ for every $\lambda\in\Lambda,$ where ${2}\leqslant{|}\Lambda{|}\leqslant\omega.$
   Suppose also that for each finite nonempty ${I}\subseteq\Lambda,$

   \begin{itemize}
   \item[\ding{226}]
      if ${R}_{{i}}\in\riseee_{\mathbf{H}({i})}({X}_{{i}})$ for all ${i}\in{I},$
   \item[\ding{226}]
      then $\bigcap_{{i}\in{I}}{R}_{{i}}$ is infinite.
   \end{itemize}
   Then the product $\prod_{\lambda\in\Lambda}{X}_{\lambda}$ has a $\pi\!$-tree.
\end{teo}

\begin{sle}\label{cor.Y*product}
    Suppose that $\mathbf{H}(\lambda)$ is a $\pi\!$-tree on a space ${X}_{\lambda}$
    and $\cofin\omega\gg\riseee_{\mathbf{H}(\lambda)}({X}_{\lambda})$ for all $\lambda\in\Lambda,$ where ${1}\leqslant{|}\Lambda{|}\leqslant\omega.$
    Suppose also that a space ${Y}$ has a $\pi\!$-tree. Then the product ${Y}\times\prod_{\lambda\in\Lambda}{X}_{\lambda}$ also has a $\pi\!$-tree.
\end{sle}

\begin{proof}[\textbf{\textup{Proof of Corollary~\ref{cor.Y*product}}}]
   Let $\mathbf{G}$ be a $\pi\!$-tree on ${Y}\hspace{-1pt}$ and ${I}\subseteq\Lambda$ be finite and nonempty.
   Now, if ${R}\in\riseee_{\mathbf{G}}({Y})$  and ${R}_{{i}}\in\riseee_{\mathbf{H}({i})}({X}_{{i}})$
   for every ${i}\in{I},$ then ${R}\in[\omega]^{\omega}$ by (b3) of Lemma~\ref{rem.rise.vs.grow.into} and it follows from (a) of Lemma~\ref{rem.rise.vs.grow.into} that $\bigcap_{{i}\in{I}}{R}_{{i}}\supseteq\omega\setminus{n}$ for some ${n}\in\omega.$ Therefore
   ${R}\cap\bigcap_{{i}\in{I}}{R}_{{i}}$ is infinite.
\end{proof}

\begin{proof}[\textbf{\textup{Proof of Theorem~\ref{th.1}}}]

   We may assume that    ${2}\leqslant\Lambda\in\omega\hspace{2pt}{\cup}\hspace{1pt}\{\omega\}.$
   By (c) of Lemma~\ref{lem.F.vs.S.by.isomorph}, for each ${n}\in\Lambda,$
   there exist an isomorphism $(\varphi_{{n}},\psi_{{n}})$ between $\mathbf{H}({n})$ and $\mathbf{S}$ and a topology $\tau_{{n}}$ on ${}^{\omega}\omega$ such that $\psi_{{n}}$ is a homeomorphism from ${X}_{{n}}$ onto $({}^{\omega}\omega,\tau_{{n}})$ and $\mathbf{S}$ is a $\pi\!$-tree on $({}^{\omega}\omega,\tau_{{n}}).$
   It follows that
   $$
   \riseee_{\mathbf{H}({n})}({X}_{{n}})=\riseee_{\mathbf{S}}\big(({}^{\omega}\omega,\tau_{{n}})\big)
   \qquad\text{for all}\enskip{n}\in\Lambda.
   $$
   Now, for every ${k}\in\Lambda,$ we have the following:
   \begin{equation}\label{*reform.1}
      \textstyle\text{if }\quad{R}_{{i}}\in\riseee_{\mathbf{S}}
      \big(({}^{\omega}\omega,\tau_{{i}})\big)
      \enskip\text{for every}\enskip{i}\in{k}+{1},
      \quad\text{ then }\quad\bigcap_{{i}\in{k}+{1}}{R}_{{i}}\enskip\text{is infinite.}
   \end{equation}
   And we must prove that the space $\prod_{{n}\in\Lambda}({}^{\omega}\omega,\tau_{{n}})$ has a $\pi\!$-tree.
   %   And we must prove that the space
   %   $\big({}^{\Lambda}({}^{\omega}\omega),\tau\big)\coloneq
   %   \prod_{{n}\in\Lambda}({}^{\omega}\omega,\tau_{{n}})$ (that is, $\tau$ is the Tychonoff product topology) has a $\pi\!$-tree.

   In this proof we use several specific notations.
   First, ${E}\cdot{F}\coloneq\{\hspace{1pt}{e}\,\hspace{0.5pt}{\cup}\,{f}:
   {e}\,\hspace{0.5pt}{\in}\hspace{0.5pt}\hspace{1pt}{E},
   {f}\hspace{1.1pt}{\in}\hspace{1.5pt}{F}\hspace{1.4pt}\}.$
   We use this operation in situations when
   %${E}\subseteq{}^{{A}}{C},$ ${F}\subseteq{}^{{B}}{C},$ and ${A}\cap{B}=\varnothing,$
   ${E}\subseteq{}^{{A}}{C}$ and ${F}\subseteq{}^{{B}}{C}$
   with ${A}\cap{B}=\varnothing,$
   so that
   $$
   {E}\cdot{F}\ =
   \ \big\{\:{p}\in{}^{{A}\hspace{0.3pt}\cup{B}}{C}\::\:{p}{\upharpoonright}{A}\in{E}
   \enskip\textsf{and}\enskip{p}{\upharpoonright}{B}\in{F}\:\big\};
   $$
   in particular, when ${B}=\varnothing,$ we have ${E}\cdot{}^{\varnothing}{C}={E}$
   because ${}^{\varnothing}{C}=\{\varnothing\}.$
   %   If ${E}\subseteq{}^{{A}}{C},$ ${F}\subseteq{}^{{B}}{C},$ and ${A}\cap{B}=\varnothing,$ then
   %   $$
   %   {E}\cdot{F}\ \coloneq
   %   \ \big\{\:{p}\in{}^{{A}\cup{B}}{C}\::\:{p}{\upharpoonright}{A}\in{E}
   %   \enskip\textsf{and}\enskip{p}{\upharpoonright}{B}\in{F}\:\big\};
   %   $$
   %   in particular, when ${B}=\varnothing,$ we have ${E}\cdot{}^{\varnothing}{C}={E}.$
   Recall that %As usual,
   $$
   \textstyle\prod_{{i}\in{I}}{D}_{{i}}\ \coloneq
   \ \big\{\:\langle{p}_{{i}}\rangle_{{i}\in{I}}\in{}^{{I}}
   ({\textstyle\bigcup_{{i}\in{I}}{D}_{{i}}})\;:\;{p}_{{i}}\in{D}_{{i}}
   \,\text{ for all }\,{i}\in{I}\:\big\}.
   $$
   When ${v}\in{}^{{<}\omega}\omega$ and ${m}\in\omega,$
   we put
   \begin{equation}\label{*reform.2'}
      \widetilde{\mathbf{S}}^{{m}}_{{v}}\:\coloneq
      \;{\textstyle\bigcup}\{\mathbf{S}_{{v}
      \hspace{0.4pt}\hat{\:}\langle{l}\rangle}:{l}\in\omega\setminus{m}\}.
   \end{equation}
   Note that $\{\hspace{0.7pt}\widetilde{\mathbf{S}}^{{m}}_{{v}}:{m}\in\omega\}\gg\shoot_\mathbf{S}({v})$
   for all ${v}\in{}^{{<}\omega}\omega.$

   We build a $\pi\!$-tree on the space $\prod_{{n}\in\Lambda}({}^{\omega}\omega,\tau_{{n}})= \big({}^{\Lambda}({}^{\omega}\omega),\tau\big),$ where $\tau$ is the Tychonoff product topology, by using Lemma~\ref{lem.a(n,v,i)}. This lemma states that there
   exists an indexed family
   $$
   \big\langle\;{a}({n},{v},{i})\ :\
   {n}\,{\in}\,\omega,\ {v}\,{\in}\,{}^{{2}{n}}\omega,\ {i}\,{\in}\,
   \Lambda\hspace{1pt}{\cap}\hspace{1pt}({n}\hspace{0.7pt}{+}\hspace{0.7pt}{1})\;\big\rangle
   \vspace{-6pt}
   $$%\qquad\text{such that}$$
   such that% satisfies the following properties:

   \begin{itemize}
    \item [\textup{(a1)}]
      $\mathsurround=0pt
      \forall{n}\hspace{1pt}{\in}\hspace{1pt}\omega
      \ \forall{v}\hspace{1pt}{\in}\hspace{1pt}{}^{{2}{n}}\omega
      \ \forall{i}\hspace{1pt}{\in}\hspace{1pt}\Lambda\hspace{1pt}
      {\cap}\hspace{1pt}({n}\hspace{0.7pt}{+}\hspace{0.7pt}{1})
      \ \big[{a}({n},{v},{i})\in{}^{{n}}\omega\big];$
   \item [\textup{(a2)}]
      $\mathsurround=0pt
      \forall{n}\hspace{1pt}{\in}\hspace{1pt}\omega
      \ \forall{v}\hspace{1pt}{\in}\hspace{1pt}{}^{{2}{n}}\omega
      \ \forall{m}\hspace{1pt}{\in}\hspace{1pt}\omega$

      $\displaystyle
      \Big(
      \big(\prod_{{i}\in\Lambda\cap({n}+{1})}\!\!\!\!\!\widetilde{\mathbf{S}}^{{m}}_{{a}({n},{v},{i})}\big)
      \setminus
      \big(\prod_{{i}\in\Lambda\cap({n}+{1})}\!\!\!\!\!\widetilde{\mathbf{S}}^{{m}+{1}}_{{a}({n},{v},{i})}\big)
      \Big)
      \cdot\,
      {}^{\Lambda\cap\{{n}+{1}\}}({}^{\omega}\omega)
      \ \equiv
      \ \,\bigsqcup_{{l}\in\omega}
      \ \prod_{{i}\in\Lambda\cap({n}+{2})}\!\!\!\!\!
      \mathbf{S}_{{a}({n}+{1},{v}\hspace{0.4pt}\hat{\:}\langle{m},{l}\rangle,{i})}.$
   \end{itemize}
   Let $\mathbf{G}(\Lambda)$ be a foliage tree with
   $\skeleton\mathbf{G}(\Lambda)\coloneq({}^{{<}\omega}\omega,{\subset})$
   and with leaves defined as follows:

   \begin{itemize}
    \item [\textup{(b1)}]
      $\mathsurround=0pt
      \forall{n}\,{\in}\,\omega
      \,\ \forall{v}\,{\in}\,{}^{{2}{n}}\omega$

      $\mathsurround=0pt \displaystyle
      \mathbf{G}(\Lambda)_{{v}}\;\coloneq
      \ \big(\!\!\!\!\prod_{\ {i}\in\Lambda\cap({n}+{1})}\!\!\!\!\!
      \mathbf{S}_{{a}({n},{v},{i})}\big)
      \cdot
      {}^{\Lambda\setminus({n}+{1})}({}^{\omega}\omega);$
   \item [\textup{(b2)}]
      $\mathsurround=0pt
      \forall{n}\,{\in}\,\omega
      \,\ \forall{v}\,{\in}\,{}^{{2}{n}}\omega
      \,\ \forall{m}\,{\in}\,\omega$

      $\mathsurround=0pt \displaystyle
      \mathbf{G}(\Lambda)_{{v}\hspace{0.4pt}\hat{\:}\langle{m}\rangle}\;\coloneq
      \ \Big(
      \big(\!\!\!\prod_{\ {i}\in\Lambda\cap({n}+{1})}\!\!\!\!\!
      \widetilde{\mathbf{S}}^{{m}}_{{a}({n},{v},{i})}\big)
      \setminus
      \big(\!\!\!\prod_{\ {i}\in\Lambda\cap({n}+{1})}\!\!\!\!\!
      \widetilde{\mathbf{S}}^{{m}+{1}}_{{a}({n},{v},{i})}\big)
      \Big)
      \cdot
      {}^{\Lambda\setminus({n}+{1})}({}^{\omega}\omega).$
   \end{itemize}
   Notice that the construction of $\mathbf{G}(\Lambda)$ doesn't depend on topologies $\tau_{{n}},$ ${n}\in\Lambda;$ it depends only on the cardinality of $\Lambda.$

   To complete the proof, we show that $\mathbf{G}(\Lambda)$ is indeed a $\pi\!$-tree on $\big({}^{\Lambda}({}^{\omega}\omega),\tau\big):$

   \begin{itemize}
   \item[\ding{42}]
      $\mathsurround=0pt
      \mathbf{G}(\Lambda)\,$ is a foliage $\omega,\omega\!$-tree.
   \item[\ding{42}]
      $\mathsurround=0pt
      \mathbf{G}(\Lambda)_{{0}_{\mathbf{G}(\Lambda)}}={}^{\Lambda}({}^{\omega}\omega).$
   \end{itemize}
   We have ${0}_{\mathbf{G}(\Lambda)}=\langle\rangle,$ clause (b1) with ${n}={0}$ says that $$
   \mathbf{G}(\Lambda)_{\langle\rangle}\ =
   \ {}^{\{0\}}\mathbf{S}_{{a}({0},\langle\rangle,{0})}
   \cdot
   {}^{\Lambda\setminus{1}}({}^{\omega}\omega),
   $$
   so using (\ref{*(c1)}) (see the proof of Lemma~\ref{lem.a(n,v,i)}) we have
   $$
   \mathbf{G}(\Lambda)_{{0}_{\mathbf{G}(\Lambda)}}\ =
   \ {}^{\{0\}}\mathbf{S}_{\langle\rangle}
   \cdot
   {}^{\Lambda\setminus{1}}({}^{\omega}\omega)\ =
   \ {}^{\{0\}}({}^{\omega}\omega)
   \cdot
   {}^{\Lambda\setminus\{{0}\}}({}^{\omega}\omega)\ =
   \ {}^{\Lambda}({}^{\omega}\omega).$$

   \begin{itemize}
   \item[\ding{42}]
      $\mathsurround=0pt
      \mathbf{G}(\Lambda)\,$ is open in $\big({}^{\Lambda}({}^{\omega}\omega),\tau\big).$
   \end{itemize}
   By (b) of Corollary~\ref{cor.scope.in.Baire.f.tree}, every set $\mathbf{S}_{{v}}$ is closed-and-open in each of spaces $({}^{\omega}\omega,\tau_{{n}}),$ and the formula
   $$
   \widetilde{\mathbf{S}}^{{m}}_{{v}}\ =\ \mathbf{S}_{{v}}\setminus
   {\textstyle\bigcup}\{\mathbf{S}_{{v}\hspace{0.4pt}\hat{\:}\langle{l}\rangle}:{l}\in{m}\}
   $$
   (which follows from (\ref{*reform.2'}))
   implies that  every set $\widetilde{\mathbf{S}}^{{m}}_{{v}}$ is closed-and-open in each of $({}^{\omega}\omega,\tau_{{n}})$ too.
   Therefore every leaf of $\mathbf{G}(\Lambda)$ is open in $\big({}^{\Lambda}({}^{\omega}\omega),\tau\big).$

   \begin{itemize}
   \item[\ding{42}]
      $\mathsurround=0pt
      \mathbf{G}(\Lambda)\,$ is locally strict.
   \end{itemize}
   Let ${t}\in\nodes\mathbf{G}(\Lambda).$
   First, suppose that ${t}\in{}^{{2}{n}}\omega$ for some ${n}\in\omega.$
   Since $\mathbf{S}_{{u}}=\widetilde{\mathbf{S}}^{{0}}_{{u}}$ for all ${u}\in{}^{{<}\omega}\omega,$ then by (b1) we have
   $$
   \mathbf{G}(\Lambda)_{{t}}\ =
   \ \big(\prod_{{i}\in\Lambda\cap({n}+{1})}\widetilde{\mathbf{S}}^{{0}}_{{a}({n},{t},{i})}\big)
   \cdot
   {}^{\Lambda\setminus({n}+{1})}({}^{\omega}\omega).
   $$
   Note that for each ${u}\in{}^{{<}\omega}\omega,$ the $\langle\hspace{0.5pt}\widetilde{\mathbf{S}}^{{m}}_{{u}}\rangle_{{m}\in\omega}$ is a strictly decreasing sequence of sets and
   $\bigcap_{{m}\in\omega}\widetilde{\mathbf{S}}^{{m}}_{{u}}=\varnothing.$
   Then it follows from (b2) that
   $$
   \mathbf{G}(\Lambda)_{{t}}\ \equiv
   \ \bigsqcup_{{m}\in\omega}\mathbf{G}(\Lambda)_{{t}\hspace{0.5pt}\hat{\:}\langle{m}\rangle}.
   $$

   Now suppose that ${t}\in{}^{{2}{n}+{1}}\omega$ for some ${n}\in\omega,$ so that ${t}={u}\hspace{0.5pt}\hat{\:}\langle{m}\rangle$ for some ${u}\in{}^{{2}{n}}\omega,$ ${m}\in\omega.$
   Then by (b2) we have
   $$
   \mathbf{G}(\Lambda)_{{t}}\ =
   \ \mathbf{G}(\Lambda)_{{u}\hspace{0.4pt}\hat{\:}\langle{m}\rangle}\ =
   \ \Big(
   \big(\!\!\!\prod_{\ {i}\in\Lambda\cap({n}+{1})}\!\!\!\!
   \widetilde{\mathbf{S}}^{{m}}_{{a}({n},{u},{i})}\big)
   \setminus
   \big(\!\!\!\prod_{\ {i}\in\Lambda\cap({n}+{1})}\!\!\!\!
   \widetilde{\mathbf{S}}^{{m}+{1}}_{{a}({n},{u},{i})}\big)
   \Big)
   \cdot
   {}^{\Lambda\cap\{{n}+{1}\}}({}^{\omega}\omega)
   \cdot
   {}^{\Lambda\setminus({n}+{2})}({}^{\omega}\omega),
   $$
   so (a2) implies
   $$
   \mathbf{G}(\Lambda)_{{t}}
   \ \equiv
   \ \bigsqcup_{{l}\in\omega}
   \Big(
   \big(\!\!\!\prod_{\ {i}\in\Lambda\cap({n}+{2})}\!\!\!\!
   \mathbf{S}_{{a}({n}+{1},{u}\hspace{0.4pt}\hat{\:}\langle{m},{l}\rangle,{i})}\big)
   \cdot
   {}^{\Lambda\setminus({n}+{2})}({}^{\omega}\omega)
   \Big),
   $$
   and then using (b1) with
   ${v}={u}\hspace{0.5pt}\hat{\:}\langle{m},{l}\rangle\in{}^{{2}({n}+{1})}\omega$
   we have
   $$
   \mathbf{G}(\Lambda)_{{t}}\ \equiv
   \ \bigsqcup_{{l}\in\omega}
   \mathbf{G}(\Lambda)_{{u}\hspace{0.4pt}\hat{\:}\langle{m},{l}\rangle},
   \quad\text{ that is, }\quad
   \mathbf{G}(\Lambda)_{{t}}\ \equiv
   \ \bigsqcup_{{l}\in\omega}
   \mathbf{G}(\Lambda)_{{t}\hspace{0.5pt}\hat{\:}\langle{l}\rangle}.
   $$

   \begin{itemize}
   \item[\ding{42}]
      $\mathsurround=0pt
      \mathbf{G}(\Lambda)\,$ has strict branches.
   \end{itemize}
   Suppose that ${B}$ is a branch in $\skeleton\mathbf{G}(\Lambda),$ which means that ${B}=\{{z}{\upharpoonright}\hspace{0.4pt}{n}:{n}\in\omega\}$ for some ${z}\in{}^{\omega}\omega.$
   Since $\mathbf{G}(\Lambda)$ is a locally strict foliage $\omega,\omega\!$-tree,
   then $\mathbf{G}(\Lambda)$ is nonincreasing, so
   \begin{equation}\label{*55.3}
      \bigcap_{{b}\in{B}}\mathbf{G}(\Lambda)_{{b}}\ =
      \ \bigcap_{{n}\in\omega}\mathbf{G}(\Lambda)_{{z}{\upharpoonright}\hspace{0.2pt}{2}{n}}
   \end{equation}
   because the chain $\{{z}{\upharpoonright}\hspace{0.5pt}{2}{n}:{n}\in\omega\}$ is cofinal in $({B},{\subset}).$
   By (b1) we have
   \begin{equation}\label{*55.2}
      \mathbf{G}(\Lambda)_{{z}{\upharpoonright}\hspace{0.3pt}{2}{n}}\ =
      \ \big(\!\!\!\prod_{\ {i}\in\Lambda\cap({n}+{1})}\!\!\!\!
      \mathbf{S}_{{a}({n},{{z}{\upharpoonright}\hspace{0.3pt}{2}{n}},{i})}\big)
      \cdot
      {}^{\Lambda\setminus({n}+{1})}({}^{\omega}\omega)
      \qquad\text{for all}\enskip{n}\in\omega.
   \end{equation}
   Since $\mathbf{G}(\Lambda)$ is nonincreasing, it follows from (\ref{*55.2}) and (a1) that
   $$
   \mathbf{S}_{{a}({n},{{z}{\upharpoonright}\hspace{0.3pt}{2}{n}},{i})}\supset
   \mathbf{S}_{{a}({n}+{1},{{z}{\upharpoonright}\hspace{0.3pt}{2}({n}+{1})},{i})}
   \qquad\text{for all }\;{n}\in\omega\;\text{ and }\;{i}\in\Lambda\cap({n}+{1})
   $$
   --- that is, for all ${i}\in\Lambda$ and ${n}\in\omega\setminus{i}.$
   This implies
   $$
   {a}({n},{{z}{\upharpoonright}\hspace{0.4pt}{2}{n}},{i})\subset
   {a}\big({n}{+}{1},{{z}{\upharpoonright}\hspace{0.4pt}{2}({n}{+}{1})},{i}\big)
   \qquad\text{for all }\;{i}\in\Lambda\;\text{ and }\;{n}\in\omega\setminus{i},
   $$
   and then, for every ${i}\in\Lambda,$ there is ${y}_{{i}}\in{}^{\omega}\omega$ such that ${a}({n},{{z}{\upharpoonright}\hspace{0.4pt}{2}{n}},{i})\subset{y}_{{i}}$ for all ${n}\in\omega\setminus{i}.$
   Then
   \begin{equation}\label{*63.1}
      \bigcap_{{n}\in\omega\setminus{i}}\mathbf{S}_{{a}({n},{{z}{\upharpoonright}\hspace{0.3pt}{2}{n}},{i})}=
      \{{y}_{{i}}\}\qquad\text{for all }\;{i}\in\Lambda.
   \end{equation}
   Put ${y}\coloneq\langle{y}_{{i}}\rangle_{{i}\in\Lambda}\in{}^{\Lambda}({}^{\omega}\omega).$
   Now (\ref{*55.2}) and (\ref{*63.1}) imply
   $\bigcap_{{n}\in\omega}\mathbf{G}(\Lambda)_{{z}{\upharpoonright}\hspace{0.3pt}{2}{n}}=\{{y}\},$
   so the $\bigcap_{{b}\in{B}}\mathbf{G}(\Lambda)_{{b}}$ is a singleton by (\ref{*55.3}).

   \begin{itemize}
   \item[\ding{42}]
      $\mathsurround=0pt
      \mathbf{G}(\Lambda)\,$ grows into $\big({}^{\Lambda}({}^{\omega}\omega),\tau\big).$
   \end{itemize}
   Suppose that
   $$
   {p}=\langle{p}_{{i}}\rangle_{{i}\in\Lambda}\in{}^{\Lambda}({}^{\omega}\omega)
   \qquad\text{and}\qquad {U}\in\nbhds\Big({p},\big({}^{\Lambda}({}^{\omega}\omega),\tau\big)\Big).
   $$
   We may assume that
   \begin{equation}\label{*56.-1}
      {U}\ =\ \big(\prod_{{i}\in{k}+{1}}{U}_{{i}}\big)\cdot      {}^{\Lambda\setminus{({k}+{1})}}({}^{\omega}\omega)
   \end{equation}
   for some ${k}\in\Lambda$ and some ${U}_{{i}}\in\nbhds\big({p}_{{i}},({}^{\omega}\omega,\tau_{{i}})\big)$
   for every ${i}\in{k}+{1}.$ Put ${R}_{{i}}\coloneq\rise_{\mathbf{S}}({p}_{{i}},{U}_{{i}})$ for every ${i}\in{k}+{1}.$
   Then $\bigcap_{{i}\in{k}+{1}}{R}_{{i}}$ is infinite by (\ref{*reform.1}),
   so there is some ${\bar{n}}\in\bigcap_{{i}\in{k}+{1}}{R}_{{i}}$ such that ${\bar{n}}\geqslant{k}.$
   By definition of $\rise_{\mathbf{S}}({p}_{{i}},{U}_{{i}}),$ for each ${i}\in{k}+{1},$ there is ${v}_{{i}}\in\scope_{\mathbf{S}}({p}_{{i}})$
   such that
   $$
   \height_{\mathbf{S}}({v}_{{i}})={\bar{n}}
   \qquad\text{and}\qquad
   \shoot_{\mathbf{S}}({v}_{{i}})\gg\{{U}_{{i}}\}.
   $$
   This means that for each ${i}\in{k}+{1},$ we have ${p}_{{i}}\in\mathbf{S}_{{v}_{{i}}},$ ${v}_{{i}}\in{}^{{\bar{n}}}\omega$ (hence ${v}_{{i}}={p}_{{i}}{\upharpoonright}\hspace{0.6pt}{\bar{n}}$), and
   $\widetilde{\mathbf{S}}^{{m}_{{i}}}_{{v}_{{i}}}=
   \widetilde{\mathbf{S}}^{{m}_{{i}}}_{{p}_{{i}}{\upharpoonright}\hspace{0.5pt}{\bar{n}}}\subseteq{U}_{{i}}$ for some ${m}_{{i}}\in\omega.$
   Let ${\bar{m}}$ be the maximal element of $\{{m}_{{i}}:{i}\in{k}+{1}\}.$ Then
   $\widetilde{\mathbf{S}}^{{\bar{m}}}_{{p}_{{i}}{\upharpoonright}\hspace{0.5pt}{\bar{n}}}
   \subseteq{U}_{{i}}$
   for all ${i}\in{k}+{1},$ and hence
   $$
   \prod_{{i}\in{k}+{1}}\widetilde{\mathbf{S}}^{{\bar{m}}}_{{p}_{{i}}{\upharpoonright}\hspace{0.5pt}{\bar{n}}}
   \ \subseteq\ \prod_{{i}\in{k}+{1}}{U}_{{i}}.
   $$
   Note that ${k}+{1}=\Lambda\cap({k}+{1})$ because ${k}\in\Lambda,$ so using (\ref{*56.-1}) we get
   \begin{equation}\label{*56.1}
      \big(\!\!\!\prod_{\ {i}\in\Lambda\cap{({k}+{1})}}\!\!\!\!
      \widetilde{\mathbf{S}}^{{\bar{m}}}_{{p}_{{i}}{\upharpoonright}\hspace{0.5pt}{\bar{n}}}\big)
      \cdot
      {}^{\Lambda\setminus{({k}+{1})}}({}^{\omega}\omega)
      \ \subseteq\ {U}.
   \end{equation}
   We already know that $\mathbf{G}(\Lambda)$ is a Baire foliage tree on $\big({}^{\Lambda}({}^{\omega}\omega),\tau\big),$ so using (d) of Corollary~\ref{cor.scope.in.Baire.f.tree} we can take node ${\bar{v}}\in\scope_{\mathbf{G}(\Lambda)}({p})$
   such that ${\bar{v}}\in{}^{{2}{\bar{n}}}\omega.$
   Then, using (b1) with ${n}={\bar{n}}$ and ${v}={\bar{v}},$ we have
   $$
   {p}\ =\ \langle{p}_{{i}}\rangle_{{i}\in\Lambda}\ \in\ \mathbf{G}(\Lambda)_{{\bar{v}}}\ =
   \ \big(\!\!\!\prod_{\ {i}\in\Lambda\cap({\bar{n}}+{1})}\!\!\!\!
   \mathbf{S}_{{a}({\bar{n}},{\bar{v}},{i})}\big)
   \cdot{}^{\Lambda\setminus({\bar{n}}+{1})}({}^{\omega}\omega),
   $$
   so ${p}_{{i}}\in\mathbf{S}_{{a}({\bar{n}},{\bar{v}},{i})}$ for all ${i}\in\Lambda\cap({\bar{n}}+{1}),$
   and hence using (a1) we get ${a}({\bar{n}},{\bar{v}},{i})={p}_{{i}}{\upharpoonright}\hspace{0.6pt}{\bar{n}}$
   for all ${i}\in\Lambda\cap({\bar{n}}+{1}).$
   Using (b2) and inequality ${\bar{n}}\geqslant{k}$ we can write
   $$
   \mathbf{G}(\Lambda)_{{\bar{v}}\hspace{0.7pt}\hat{\:}\langle{m}\rangle}
   \ \subseteq
   \ \big(\!\!\!\prod_{\ {i}\in\Lambda\cap({\bar{n}}+{1})}\!\!\!\!
   \widetilde{\mathbf{S}}^{{m}}_{{p}_{{i}}{\upharpoonright}\hspace{0.5pt}{\bar{n}}}\big)
   \cdot
   {}^{\Lambda\setminus({\bar{n}}+{1})}({}^{\omega}\omega)
   \ \subseteq
   \ \big(\!\!\!\prod_{\ {i}\in\Lambda\cap{({k}+{1})}}\!\!\!\!
   \widetilde{\mathbf{S}}^{{m}}_{{p}_{{i}}{\upharpoonright}\hspace{0.5pt}{\bar{n}}}\big)
   \cdot
   {}^{\Lambda\setminus{({k}+{1})}}({}^{\omega}\omega)
   \qquad\text{for all}\enskip{m}\in\omega.
   $$
   Therefore using (\ref{*56.1}) we get $\mathbf{G}(\Lambda)_{{\bar{v}}\hspace{0.7pt}\hat{\:}\langle{m}\rangle}
   \subseteq{U}$ for all ${m}\in\omega\setminus{\bar{m}}.$ This means that we have found ${\bar{v}}\in\scope_{\mathbf{G}(\Lambda)}({p})$ such that
   $\shoot_{\mathbf{G}(\Lambda)}({\bar{v}})\gg\{{U}\}.$
\end{proof}

\begin{lem}\label{lem.a(n,v,i)}
   For each
   $\Lambda\in\big(\omega\hspace{2pt}{\cup}\hspace{1pt}\{\omega\}\big)\setminus{2},$
   there is a family
   $\big\langle\:{a}({n},{v},{i})\;:
   \;{n}\,{\in}\,\omega,\ {v}\,{\in}\,{}^{{2}{n}}\omega,\ {i}\,{\in}\,
   \Lambda\hspace{1pt}{\cap}\hspace{1pt}({n}\hspace{0.7pt}{+}\hspace{0.7pt}{1})\:\big\rangle$
   such that
   \begin{itemize}
    \item [\textup{(a1)}]
      $\mathsurround=0pt
      \forall{n}\hspace{1pt}{\in}\hspace{1pt}\omega
      \ \forall{v}\hspace{1pt}{\in}\hspace{1pt}{}^{{2}{n}}\omega
      \ \forall{i}\hspace{1pt}{\in}\hspace{1pt}\Lambda\hspace{1pt}
      {\cap}\hspace{1pt}({n}\hspace{0.7pt}{+}\hspace{0.7pt}{1})
      \ \big[{a}({n},{v},{i})\in{}^{{n}}\omega\big];$
   \item [\textup{(a2)}]
      $\mathsurround=0pt
      \forall{n}\hspace{1pt}{\in}\hspace{1pt}\omega
      \ \forall{v}\hspace{1pt}{\in}\hspace{1pt}{}^{{2}{n}}\omega
      \ \forall{m}\hspace{1pt}{\in}\hspace{1pt}\omega$

      $\displaystyle
      \Big(
      \big(\prod_{{i}\in\Lambda\cap({n}+{1})}\!\!\!\!\!\widetilde{\mathbf{S}}^{{m}}_{{a}({n},{v},{i})}\big)
      \setminus
      \big(\prod_{{i}\in\Lambda\cap({n}+{1})}\!\!\!\!\!\widetilde{\mathbf{S}}^{{m}+{1}}_{{a}({n},{v},{i})}\big)
      \Big)
      \cdot\,
      {}^{\Lambda\cap\{{n}+{1}\}}({}^{\omega}\omega)
      \ \equiv
      \ \,\bigsqcup_{{l}\in\omega}
      \ \prod_{{i}\in\Lambda\cap({n}+{2})}\!\!\!\!\!
      \mathbf{S}_{{a}({n}+{1},{v}\hspace{0.4pt}\hat{\:}\langle{m},{l}\rangle,{i})}.$
   \end{itemize}
\end{lem}

\begin{proof}
   We construct this indexed family by recursion on ${n}\in\omega$ as follows:

   When ${n}={0},$
   we have ${}^{{2}{n}}\omega={}^{{n}}\omega={}^{{0}}\omega=\big\{\langle\rangle\big\}$
   and $\Lambda\cap({n}+{1})=\{{0}\}$
   because $\Lambda\geqslant{2},$ so (a1) with ${n}={0}$ just says
   \begin{equation}\label{*(c1)}
      {a}\big({0},\langle\rangle,{0}\big)=\langle\rangle.
   \end{equation}

   When ${n}={1},$ we must choose ${a}({1},{v},{i})\in{}^{{1}}\omega$
   (for all ${v}\in{}^{{2}}\omega$ and ${i}\in\Lambda\cap{2}$) in such a way that (a2) with ${n}={0}$ is satisfied.  Since $\Lambda\geqslant{2},$ then $\Lambda\cap{1}=\{{0}\}$ and $\Lambda\cap{2}=\{{0},{1}\},$
   so (a2) with ${n}={0}$ says that
   $$
   \big({}^{\{0\}}\widetilde{\mathbf{S}}^{{m}}_{{a}({0},\langle\rangle,{0})}
   \setminus
   {}^{\{0\}}\widetilde{\mathbf{S}}^{{m}+{1}}_{{a}({0},\langle\rangle,{0})}\big)
   \cdot
   {}^{\{{1}\}}({}^{\omega}\omega)
   \ \equiv
   \ \,\bigsqcup_{{l}\in\omega}
   \ \prod_{{i}\in\{{0},{1}\}}\!\!
   \mathbf{S}_{{a}({1},\langle\rangle\hat{\:}\langle{m},{l}\rangle,{i})}
   \qquad\text{for all}\enskip{m}\in\omega.
   $$
   Using (\ref{*reform.2'}) and (\ref{*(c1)}), this can be simplified to
   $$
   {}^{\{0\}}\mathbf{S}_{\langle{m}\rangle}
   \cdot
   {}^{\{{1}\}}({}^{\omega}\omega)
   \ \equiv
   \ \,\bigsqcup_{{l}\in\omega}
   \ \prod_{{i}\in\{{0},{1}\}}\!\!
   \mathbf{S}_{{a}({1},\langle{m},{l}\rangle,{i})}
   \qquad\text{for all}\enskip\forall{m}\in\omega.
   $$
   Then we can take
   ${a}\big({1},\langle{m},{l}\rangle,{0}\big)\coloneq\langle{m}\rangle$ and
   ${a}\big({1},\langle{m},{l}\rangle,{1}\big)\coloneq\langle{l}\rangle$
   for every ${m},{l}\in\omega.$

   When ${n}\geqslant{2},$ the choice of ${a}({n},{v},{i})$ can be carried out similar to the case ${n}={1}$ if we note that
   $$
   {}^{\omega}\omega\:\equiv\bigsqcup_{{a}\in{}^{{h}}\omega}\mathbf{S}_{{a}}
   \qquad\text{for all }\;
   {h}\in\omega,
   $$
   and that for every ${k}\geqslant{2},$
   every ${a}=\langle{a}_{{i}}\rangle_{{i}\in{k}}\in{}^{{k}}({}^{{2}{n}}\omega),$
   and every ${m}\in\omega,$
   $$
   \big(\prod_{{i}\in{k}}\widetilde{\mathbf{S}}^{{m}}_{{a}_{{i}}}\big)
   \setminus
   \big(\prod_{{i}\in{k}}\widetilde{\mathbf{S}}^{{m}+{1}}_{{a}_{{i}}}\big)
   \ =
   \ \big(\prod_{{i}\in{k}}\:\bigcup_{{l}\in\omega\setminus{m}}
   \!\!\mathbf{S}_{{a}_{{i}}\hat{\:}\langle{l}\rangle}\big)
   \setminus
   \big(\prod_{{i}\in{k}}\!\bigcup_{\ {l}\in\omega\setminus({m}+{1})\!\!\!\!}
   \!\!\!\!\!\!\!\mathbf{S}_{{a}_{{i}}\hat{\:}\langle{l}\rangle}\big)
   \ =
   $$
   $$
   =
   \ \bigcup
   \Big\{
   \prod_{{i}\in{k}}\mathbf{S}_{{a}_{{i}}\hat{\:}\langle{l}_{{i}}\rangle}
   :
   {l}=\langle{l}_{{i}}\rangle_{{i}\in{k}}\in{}^{{k}}(\omega\setminus{m})
   \Big\}
   \setminus
   \bigcup
   \Big\{
   \prod_{{i}\in{k}}\mathbf{S}_{{a}_{{i}}\hat{\:}\langle{l}_{{i}}\rangle}
   :
   {l}=\langle{l}_{{i}}\rangle_{{i}\in{k}}\in{}^{{k}}\big(\omega\setminus({m}+{1})\big)
   \Big\}
   \ \equiv
   $$
   $$
   \equiv
   \ \bigsqcup
   \Big\{
   \prod_{{i}\in{k}}\mathbf{S}_{{a}_{{i}}\hat{\:}\langle{l}_{{i}}\rangle}
   :
   {l}=\langle{l}_{{i}}\rangle_{{i}\in{k}}\in{}^{{k}}(\omega{\setminus}{m})\setminus
   {}^{{k}}\big(\omega{\setminus}({m}{+}{1})\big)
   \Big\}
   $$
   and the set ${}^{{k}}(\omega{\setminus}{m})\setminus
   {}^{{k}}\big(\omega{\setminus}({m}{+}{1})\big)$ is infinite.
\end{proof}

\section{The second theorem}

\begin{teo}\label{th.2}\mbox{ }

   \textup{(a)} Suppose that $\mathbf{F}(\alpha)$ is a $\pi\!$-tree on a space ${X}_{\alpha}$ for every $\alpha\in{A},$ where ${1}\leqslant{|}{A}{|}\leqslant\omega.$ Suppose also that for each $\alpha\in{A},$ there is $\gamma_{\alpha}\subseteq\pset(\omega)$ such that

      \begin{itemize}
      \item [\ding{226}]
         $\mathsurround=0pt {|}\gamma_{\alpha}{|}\leqslant\omega,$
      \item [\ding{226}]
         $\mathsurround=0pt \gamma_{\alpha}\,$ has the \textup{\textsf{FIP}}, and
      \item [\ding{226}]
         $\mathsurround=0pt \gamma_{\alpha}\gg\riseee_{\mathbf{F}(\alpha)}({X}_{\alpha}).$
      \end{itemize}
   Then the product $\prod_{\alpha\in{A}}{X}_{\alpha}$ has a $\pi\!$-tree.

   \textup{(b)} Suppose, in addition to \textup{(a)}, that  $\mathbf{G}$ is a $\pi\!$-tree on a space ${Y}\hspace{-1pt}$ and $\riseee_{\mathbf{G}}({Y})$ has the \textup{\textsf{FIP}}. Then the product ${Y}\times\prod_{\alpha\in{A}}{X}_{\alpha}$ also has a $\pi\!$-tree.
\end{teo}

\begin{lem}\label{lem.sdvig.filtrov}
   Suppose that ${2}\leqslant\Lambda\in\omega\hspace{2pt}{\cup}\hspace{1pt}\{\omega\}$
   and for each ${n}\in\Lambda,$ $\varnothing\neq\delta_{{n}}\subseteq\pset(\omega)\,{\setminus}\,\{\varnothing\}$ and $\bigcap\delta_{{n}}=\varnothing.$
   Suppose also that $\delta_{{0}}$ has the \textup{\textsf{FIP}} and
   for each ${n}\in\Lambda\setminus\{{0}\},$ there is $\gamma_{{n}}\subseteq\pset(\omega)$ such that
   \begin{itemize}
   \item [\ding{226}]
      $\mathsurround=0pt {|}\gamma_{{n}}{|}\leqslant\omega,$
   \item [\ding{226}]
      $\mathsurround=0pt \gamma_{{n}}\,$ has the \textup{\textsf{FIP}}, and
   \item [\ding{226}]
      $\mathsurround=0pt \gamma_{{n}}\gg\delta_{{n}}.$
   \end{itemize}
%   \begin{itemize}
%   \item[\ding{226}]
%      $\mathsurround=0pt \forall{n}\in\Lambda\setminus\{{0}\}
%      \exists\gamma_{{n}}\subseteq\pset(\omega)$ such that
%      \begin{itemize}
%      \item [\ding{51}]
%         $\mathsurround=0pt {|}\gamma_{{n}}{|}\leqslant\omega,$
%      \item [\ding{51}]
%         $\mathsurround=0pt \gamma_{{n}}\,$ has the \textup{\textsf{FIP}}, and
%      \item [\ding{51}]
%         $\mathsurround=0pt \gamma_{{n}}\gg\delta_{{n}}.$
%      \end{itemize}
%   \end{itemize}
   Then there exists a sequence $\langle\alpha_{{n}}\rangle_{{n}\in\Lambda}$ of strictly increasing functions $\alpha_{{n}}\colon\omega\to\omega$ such that

   \begin{itemize}
   \item [\textup{(\ding{168})}]
      if ${k}\in\Lambda$ and ${A}_{{i}}\in
      \big\{\alpha_{{i}}[{D}]:{D}\in\delta_{{i}}\big\}$ for all ${i}\leqslant{k},$
      then $\bigcap_{{i}\leqslant{k}}{A}_{{i}}$ is infinite.
      %   \item [\textup{(\ding{168})}]
      %      if ${k}\in\Lambda$ and ${A}_{{i}}\in\alpha_{{i}}[[\delta_{{i}}]]$ for all ${i}\leqslant{k},$
      %      then $\bigcap_{{i}\leqslant{k}}{A}_{{i}}$ is infinite.
   \end{itemize}
\end{lem}

\begin{lem}\label{lem.perestrojka.LPB.po.sdvigu.fil'tra}
   Suppose that $\hspace{1pt}\mathbf{F}\hspace{-1pt}$ is a $\pi\!$-tree on a space ${X}\hspace{-1pt}$ and $\alpha\colon\omega\to\omega$ is a strictly increasing function.
   Then there exists a $\pi\!$-tree $\mathbf{H}\hspace{-1pt}$ on ${X}\hspace{-1pt}$ such that

   \begin{itemize}
   \item [\textup{(\ding{170})}]
      $\mathsurround=0pt
      \alpha\:\big[\rise_{\mathbf{F}}({p},{U})\big]
      \subseteq
      \rise_{\mathbf{H}}({p},{U})\quad$
      for all ${p}\in{X}\hspace{-1pt}$ and $\,{U}\in\nbhds({p},{X}).$
   \end{itemize}
\end{lem}

\begin{proof}[\textbf{\textup{Proof of Theorem~\ref{th.2}}}]
   Note that part (a) follows from part (b). Indeed, let $\beta\in{A},$ $\mathbf{G}\coloneq\mathbf{F}(\beta),$ and ${Y}\coloneq{X}_{\beta}.$
   Since $\gamma_{\beta}$ has the \textup{\textsf{FIP}}, $\gamma_{\beta}\gg\riseee_{\mathbf{G}}({Y}),$ and (by (a) of Lemma~\ref{rem.rise.vs.grow.into}) $\varnothing\nin\riseee_{\mathbf{G}}({Y}),$ then
   $\riseee_{\mathbf{G}}({Y})$ also has the \textup{\textsf{FIP}}. The case when ${|}{A}{|}={1}$ is trivial, so we may assume that ${A}\setminus\{\beta\}\neq\varnothing,$ and then the space
   $$
   \prod_{\alpha\in{A}}{X}_{\alpha}\ =
   \ {Y}\times\!\!\!\!\prod_{\ \alpha\in{A}\setminus\{\beta\}}\!\!\!\!\!{X}_{\alpha}
   $$
   has a $\pi\!$-tree by (b).

   To prove (b) it is convenient to assume that ${A}=\Lambda\setminus\{{0}\}$ and
   ${2}\leqslant\Lambda\in\omega\hspace{2pt}{\cup}\hspace{1pt}\{\omega\}.$
   Put ${X}_{{0}}\coloneq{Y}\hspace{-1pt}$ and $\mathbf{F}({0})\coloneq\mathbf{G};$ then we must prove that the space
   $\prod_{{n}\in\Lambda}{X}_{{n}}$
   has a $\pi\!$-tree. Let $\delta_{{n}}\coloneq\riseee_{\mathbf{F}({n})}({X}_{{n}})$ for every ${n}\in\Lambda.$
   Then using Lemma~\ref{rem.rise.vs.grow.into} we see that $\delta_{{n}}\subseteq\pset(\omega)\,{\setminus}\,\{\varnothing\}$
   and $\bigcap\delta_{{n}}=\varnothing$
   for all ${{n}}\in\Lambda,$ so we can apply Lemma~\ref{lem.sdvig.filtrov}.
   Then we get a sequence $\langle\alpha_{{n}}\rangle_{{n}\in\Lambda}$ of strictly increasing functions $\alpha_{{n}}\colon\omega\to\omega$
   such that condition (\ding{168}) holds.
   Next, applying Lemma~\ref{lem.perestrojka.LPB.po.sdvigu.fil'tra} to
   $\mathbf{F}_{\hspace{-1.2pt}{n}},$${X}_{{n}},$ and $\alpha_{{n}}$ for every ${n}\in\Lambda,$ we obtain a sequence
   $\big\langle\mathbf{H}({n})\big\rangle_{{n}\in\Lambda}$ such that for every ${{n}\in\Lambda},$
   $\mathbf{H}({n})$ is a $\pi\!$-tree on ${X}_{{n}}$ and
   \begin{equation}\label{*35.1}
      \alpha_{{n}}\:\big[\rise_{\mathbf{F}({n})}({p},{U})\big]
      \:\subseteq\:
      \rise_{\mathbf{H}({n})}({p},{U})
      \qquad\text{for all}\enskip{p}\in{X}_{{n}}\enskip\text{and}\enskip{U}\in\nbhds({p},{X}_{{n}}).
   \end{equation}

   Now we can use Theorem~\ref{th.1} to show that the product $\prod_{{n}\in\Lambda}{X}_{{n}}$ has a $\pi\!$-tree.
   Suppose that ${I}\subseteq\Lambda$ is finite and nonempty; then ${I}\subseteq{k}+{1}$ for some ${k}\in\Lambda.$ Let ${R}_{{i}}\in\riseee_{\mathbf{H}({i})}({X}_{{i}})$
   for every ${i}\in{k}+{1};$ we must show that the set $\bigcap_{{i}\in{I}}{R}_{{i}}$ is infinite.
   For each ${i}\in{k}+{1},$ ${R}_{{i}}=\rise_{\mathbf{H}({i})}({p}_{{i}},{U}_{{i}})$ for some ${p}_{{i}}\in{X}_{{i}}$ and ${U}_{{i}}\in\nbhds({p}_{{i}},{X}_{{i}}).$
   Put
   $
   {A}_{{i}}\coloneq\alpha_{{i}}
   \:\big[\rise_{\mathbf{F}({i})}({p}_{{i}},{U}_{{i}})\big]
   $
   for every ${i}\in{k}+{1};$
   then by (\ref{*35.1}) we have ${A}_{{i}}\subseteq{R}_{{i}}.$
   Now,
   $$
   {A}_{{i}}\,\in\,
   \big\{\alpha_{{i}}[{R}]:{R}\in\riseee_{\mathbf{F}({i})}({X}_{{i}})\big\}
   \,=\,\big\{\alpha_{{i}}[{D}]:{D}\in\delta_{{i}}\big\}
   \qquad\text{for all}\enskip{i}\in{k}+{1},
   $$
   so (by (\ding{168}) of Lemma~\ref{lem.sdvig.filtrov}) the $\bigcap_{{i}\in{k}+{1}}{A}_{{i}}$ is infinite,
   hence the $\bigcap_{{i}\in{k}+{1}}{R}_{{i}}$ is infinite,
   and then the $\bigcap_{{i}\in{I}}{R}_{{i}}$ is infinite too.
\end{proof}

\begin{proof}[\textbf{\textup{Proof of Lemma~\ref{lem.sdvig.filtrov}}}]
   It is not hard to show that each $\gamma_{{n}}$ is not empty, so we may assume that
   $\gamma_{{n}}=\big\{{G}_{{i}}({n}):{i}\in\omega\big\}$ for every ${n}\in\Lambda\setminus\{{0}\}.$
   Since $\delta_{{0}}$ has the \textup{\textsf{FIP}} and $\bigcap\delta_{{0}}=\varnothing,$
   then
   \begin{equation}\label{*a0}
      \textstyle\bigcap\varepsilon\enskip
      \text{is infinite}
      \qquad\text{for all}\enskip\varepsilon\in[\delta_{{0}}]^{{<}\omega}\,{\setminus}\,\{\varnothing\}.
   \end{equation}
   Also $\bigcap\gamma_{{n}}=\varnothing$ for all ${n}\in\Lambda\setminus\{{0}\}$
   (because $\gamma_{{n}}\gg\delta_{{n}},$ $\varnothing\nin\delta_{{n}}\neq\varnothing,$ and $\bigcap\delta_{{n}}=\varnothing$), so by the same reasons we get
   \begin{equation}\label{*a1}
      \bigcap_{{j}\leqslant{i}}{G}_{{j}}({n})\enskip
      \text{is infinite}
      \qquad\text{for all}\enskip{n}\in\Lambda\setminus\{{0}\}\enskip\text{and}\enskip{i}\in\omega.
   \end{equation}
   Now, using (\ref{*a1}), for every ${n}\in\Lambda\setminus\{{0}\}$ and ${i}\in\omega,$ we can choose
   ${f}_{{i}}({n})\in\bigcap_{{j}\leqslant{i}}{G}_{{j}}({n})$ in such a way that
   \begin{equation}\label{*a2}
      {f}_{{i}+{1}}({n})>{f}_{{i}}({n})
      \qquad\text{for all}\enskip{n}\in\Lambda\setminus\{{0}\}
      \enskip\text{and}\enskip\forall{i}\in\omega.
   \end{equation}
   Put ${F}({n})\coloneq\big\{{f}_{{i}}({n}):{i}\in\omega\big\}$ for every ${n}\in\Lambda\setminus\{{0}\};$
   then
   $\big\{{F}({n})\hspace{1pt}{\setminus}\,{m}:{m}\,{\in}\:\omega\big\}\gg\gamma_{{n}}$
   for all ${n}\in\Lambda\setminus\{{0}\},$
   and hence
   \begin{equation}\label{*a3}
      \big\{{F}({n})\setminus{m}:{m}\in\omega\big\}\:\gg\:\delta_{{n}}
      \qquad\text{for all}\enskip{n}\in\Lambda\setminus\{{0}\}.
   \end{equation}
   Let ${F}({0})\in\delta_{{0}};$ then ${F}({0})$ is infinite by (\ref{*a0}), so we may assume that
   ${F}({0})=\big\{{f}_{{i}}({0}):{i}\in\omega\big\}$
   and ${f}_{{i}+{1}}({0})>{f}_{{i}}({0})$ for all ${i}\in\omega.$
   Put ${h}_{-{1}}\coloneq-{1}$ and ${f}_{-{1}}({n})\coloneq-{1}$ for every ${n}\in\Lambda.$
   By recursion on ${i}\in\omega,$ we can build a strictly increasing sequence $\langle{h}_{{i}}\rangle_{{i}\in\omega}$ of natural numbers in such a way that
   \begin{equation}\label{*a4}
      {h}_{{i}}\;>\;{h}_{{i}-{1}}+{f}_{{i}-{j}}({j})-{f}_{{i}-{j}-{1}}({j})
      \qquad\text{for all}\enskip{i}\in\omega
      \enskip\text{and}\enskip{j}\in\Lambda\cap({i}+{1}).
   \end{equation}
   Let $\langle\beta_{{n}}\rangle_{{n}\in\Lambda}$ be a sequence of functions with
   $$
   \domain\beta_{{n}}\;\coloneq\;{F}({n})\cup\{{-}{1}\}\;=\;
   \big\{{f}_{{l}}({n}):{l}\in\omega\cup\{{-}{1}\}\big\}
   $$
   and such that
   \begin{equation}\label{*a5}
      \beta_{{n}}\big({f}_{{l}}({n})\big)={h}_{{n}+{l}}
      \qquad\text{for all}\enskip{n}\in\Lambda
      \enskip\text{and}\enskip{l}\in\omega\cup\{{-}{1}\}.
   \end{equation}
   Note that (\ref{*a5}) implies
   \begin{equation}\label{*a6}
      \beta_{{n}}\:\big[{F}({n})\big]\,=\,\{{h}_{{j}}:{j}\in\omega\setminus{n}\}
      \qquad\text{for all}\enskip{n}\in\Lambda.
   \end{equation}
   Now, for all ${n}\in\Lambda$ and ${l}\in\omega,$ (\ref{*a4}) with ${i}={n}+{l},$ ${j}={n}$ says that
   $$
   {h}_{{n}+{l}}-{h}_{{n}+{l}-{1}}\;>\;{f}_{{n}+{l}-{n}}({n})-{f}_{{n}+{l}-{n}-{1}}({n})
   \;=\;{f}_{{l}}({n})-{f}_{{l}-{1}}({n}),
   $$
   so by (\ref{*a5}) we have
   $$
   \beta_{{n}}\big({f}_{{l}}({n})\big)-\beta_{{n}}\big({f}_{{l}-{1}}({n})\big)\;>\;
   {f}_{{l}}({n})-{f}_{{l}-{1}}({n})
   \qquad\text{for all}\enskip{n}\in\Lambda
   \enskip\text{and}\enskip{l}\in\omega.
   $$
   This means that for each ${n}\in\Lambda,$ we can choose a strictly increasing function $\alpha_{{n}}\colon\omega\to\omega$ such that $\alpha_{{n}}{\upharpoonright}\hspace{0.5pt}{F}({n})=\beta_{{n}}{\upharpoonright}\hspace{0.5pt}{F}({n}).$

   Now we prove that condition (\ding{168}) is satisfied.
   Suppose that ${k}\in\Lambda$ and %${A}_{{i}}\in%\alpha_{{i}}[[\delta_{{i}}]]=
   %\big\{\alpha_{{i}}[{D}]:{D}\in\delta_{{i}}\big\}$ for all ${i}\in{k}+{1};$
   %that is,
   for every ${i}\in{k}+{1},$
   ${A}_{{i}}=\alpha_{{i}}\:\big[{D}({i})\big]$ for some ${D}({i})\in\delta_{{i}}.$
   Using (\ref{*a3}), for each ${i}\in({k}+{1})\,{\setminus}\,\{0\},$ we can choose some ${m}_{{i}}\in\omega$ such that
   ${D}({i})\supseteq{F}({i})\setminus{m}_{{i}}.$
   Therefore, by (\ref{*a6}), for each ${i}\in({k}+{1})\,{\setminus}\,\{0\},$
   we can find some ${l}_{{i}}\in\omega$
   such that ${A}_{{i}}\supseteq\{{h}_{{j}}:{j}\in\omega\setminus{l}_{{i}}\}.$
   It follows that
   $$
   \bigcap_{{i}\in({k}+{1})\setminus\{{0}\}}\!\!\!\!\!\!{A}_{{i}}
   \ \supseteq\ \{{h}_{{j}}:{j}\in\omega\setminus{l}\}
   \qquad\text{for some}\enskip{l}\in\omega.
   $$
   Now, ${D}({0})\cap{F}({0})$ is infinite by (\ref{*a0}), $\alpha_{{0}}\:\big[{F}({0})\big]=\{{h}_{{j}}:{j}\in\omega\}$ by (\ref{*a6}), and $\alpha_{{0}}$ is injective,
   therefore ${A}_{{0}}\cap\{{h}_{{j}}:{j}\in\omega\}$ is infinite.
   This means that $\bigcap_{{i}\in{k}+{1}}{A}_{{i}}$ is infinite too.
\end{proof}

\begin{proof}[\textbf{\textup{Proof of Lemma~\ref{lem.perestrojka.LPB.po.sdvigu.fil'tra}}}]
   In this proof we apply the foliage hybrid operation, see details in Section~\ref{f.h.o.}.
   Put $\alpha({-}{1})\coloneq{-}{1}.$
   Suppose that ${v}\in\nodes\mathbf{F}.$ Set
   $$
   {k}({v})\ \coloneq
   \ \alpha\big(\height_{\mathbf{F}}({v})\big)-\alpha\big(\height_{\mathbf{F}}({v})-{1}\big);
   $$
   then ${k}({v})\in\omega\,{\setminus}\,\{{0}\}.$
   Let $\mathcal{T}({v})$ be a tree isomorphic to the tree $({}^{{<}{k}({v})+{1}}\omega,\subset)$ and such that ${0}_{\mathcal{T}({v})}={v}$ and $\maxel\mathcal{T}({v})=\sons_{\mathbf{F}}({v}).$
   Let $\mathbf{G}({v})$ be a foliage tree with $\skeleton\mathbf{G}({v})\coloneq\mathcal{T}({v})$
   and with leaves defined by recursion on ${i}\in{k}({v})+{1}$ as follows:

   \begin{itemize}
   \item [\textup{(base)}]
       If ${i}={0}$ and ${t}\in\level_{\mathcal{T}({v})}\big({k}({v})-{i}\big)$ \big(that is, if ${t}\in\maxel\mathcal{T}({v})$\big),

       then $\mathbf{G}({v})_{{t}}\coloneq\mathbf{F}_{\hspace{-1.2pt}{t}}.$
   \item [\textup{(step)}]
       If ${1}\leqslant{i}\leqslant{k}({v})$ and ${t}\in\level_{\mathcal{T}({v})}\big({k}({v})-{i}\big),$

       then $\mathbf{G}({v})_{{t}}\coloneq{\displaystyle\bigcup}\big\{\mathbf{G}({v})_{{s}}:{s}\in\sons_{\mathcal{T}({v})}({t})\big\}.$
   \end{itemize}

   It is not hard to show the following (we use here the terminology of Definition~\ref{def.graft}):

   \begin{itemize}
   \item [\textup{(c1)}]
      $\mathsurround=0pt
      {0}_{\mathbf{G}({v})}={v}$ \enskip \text{and} \enskip $\maxel\mathbf{G}({v})=\sons_{\mathbf{F}}({v})=\level_{\mathbf{G}({v})}\big({k}({v})\big);$
   \item [\textup{(c2)}]
      $\mathsurround=0pt
      \mathbf{G}({v})_{{v}}=\mathbf{F}_{\hspace{-1.2pt}{v}};$
    \item [\textup{(c3)}]
      $\mathsurround=0pt
      \mathbf{G}({v})\,$ is a foliage graft for $\mathbf{F};$
   \item [\textup{(c4)}]
      $\mathsurround=0pt
      \cut\big(\mathbf{F},\mathbf{G}({v})\big)=\varnothing;$
   \item [\textup{(c5)}]
      $\mathsurround=0pt
      \mathbf{G}({v})\,$ is $\omega\!$-branching, locally strict, open in ${X},$ and
      has bounded chains;
   \item [\textup{(c6)}]
      $\mathsurround=0pt
      \height\mathbf{G}({v})={k}({v})+{1};$
   \item [\textup{(c7)}]
      $\mathsurround=0pt
      \shoot_{\mathbf{G}({v})}({t})\gg\shoot_{\mathbf{F}}({v})\quad$
      for all ${t}\in\nodes\mathbf{G}({v})\setminus\maxel\mathbf{G}({v});$
   \item [\textup{(c8)}]
      $\mathsurround=0pt
      \expl\big(\mathbf{F},\mathbf{G}({v})\big)=\varnothing.$
   \end{itemize}

   Now let $\varphi\coloneq\big\{\mathbf{G}({v}):{v}\in\nodes\mathbf{F}\big\}.$
   We may assume that
   $$\impl\mathbf{G}({v})\cap\impl\mathbf{G}({u})=\varnothing
   \qquad\text{for all }\;{v}\neq{u}\in\nodes\mathbf{F},
   $$
   so $\varphi$ is a consistent family of foliage grafts for $\mathbf{F}.$
   Let $\mathbf{H}\coloneq\fhybr(\mathbf{F},\varphi);$ note that $\loss(\mathbf{F},\varphi)=\varnothing$ by (c4). By induction on $\height_{\mathbf{F}}({v}),$ we can prove that
   \begin{equation}\label{*j0}
      \height_{\mathbf{H}}({v})\,=\,\alpha\big(\height_{\mathbf{F}}({v})-{1}\big)+{1}
      \qquad\text{for all}\enskip{v}\in\nodes\mathbf{F}.
   \end{equation}
   Indeed, if $\height_{\mathbf{F}}({v})={0},$ then ${v}={0}_{\mathbf{F}},$ so ${v}={0}_{\mathbf{H}},$ and hence
   $$
   \height_{\mathbf{H}}({v})\,=\,{0}\,=\,\alpha\big(\height_{\mathbf{F}}({v})-{1}\big)+{1}.
   $$
   If $\height_{\mathbf{F}}({v})\geqslant{1},$ then let ${t}$ be the node in $\mathbf{F}\hspace{-1pt}$ such that ${v}\in\sons_{\mathbf{F}}({t}),$ and then inductively we can write
   $$
   \height_{\mathbf{H}}({v})\ =\ \height_{\mathbf{H}}({t})+\height_{\mathbf{G}({t})}({v})\ =
   \ \height_{\mathbf{H}}({t})+\big(\height{\mathbf{G}({t})}-{1}\big)\ =
   $$
   $$
   \height_{\mathbf{H}}({t})+\big({k}({t})+{1}\big)-{1}\ =
   \ \alpha\big(\height_{\mathbf{F}}({t})-{1}\big)+{1}+{k}({t})\ =
   $$
   $$
      \alpha\big(\height_{\mathbf{F}}({t})-{1}\big)+{1}
      +\alpha\big(\height_{\mathbf{F}}({t})\big)
      -\alpha\big(\height_{\mathbf{F}}({t})-{1}\big)\ =
   $$
   $$
   \alpha\big(\height_{\mathbf{F}}({t})\big)+{1}\ =
   \ \alpha\big(\height_{\mathbf{F}}({v})-{1}\big)+{1}.
   $$
   Now, (c4)--(c6) with Lemma~\ref{lem.when.f.hybr.is.pi.tree} say that $\mathbf{H}$ is a Baire foliage tree on ${X}\hspace{-1pt}$ and (c7)--(c8) imply that each $\mathbf{G}({v})$ preserves shoots of $\mathbf{F}\hspace{-1pt}$ (see Definition~\ref{def.shoots.into}),
   so $\mathbf{H}$ grows into ${X}$ by Lemmas~\ref{lem.when.f.hyb.sprts.thrgh}  and~\ref{l.grows.into.subspace}. Therefore $\mathbf{H}$ is a $\pi\!$-tree on ${X}.$

   Let us show that (\ding{170}) holds.
   Suppose that ${p}\in{X},$ ${U}\in\nbhds({p},{X}),$ and ${r}\in\rise_{\mathbf{F}}({p},{U}).$
   Then ${r}=\height_{\mathbf{F}}({v})$ for some node ${v}\in\scope_{\mathbf{F}}({p})$
   such that $\shoot_{\mathbf{F}}({v})\gg\{{U}\}.$
   Let ${s}$ be the node in $\sons_{\mathbf{F}}({v})$ such that ${p}\in\mathbf{F}_{\hspace{-1.2pt}{s}}$ and let ${t}$ be the node in $\mathbf{G}({v})$ such that ${s}\in\sons_{\mathbf{G}({v})}({t}).$
   Then ${t}\in\scope_{\mathbf{H}}({p})$ and using (c7) and (a) of Proposition~\ref{prop.hybr} we obtain
   \begin{equation}\label{*j1}
   \shoot_{\mathbf{H}}({t})=\shoot_{\mathbf{G}({v})}({t})\gg\shoot_{\mathbf{F}}({v}),
   \end{equation}
   so $\shoot_{\mathbf{H}}({t})\gg\{{U}\},$ and hence $   \height_{\mathbf{H}}({t})\in\rise_{\mathbf{H}}({p},{U}).$
   Therefore to complete the proof it is enough to show that $\alpha({r})=\height_{\mathbf{H}}({t}).$
   Indeed, using (c6) and (\ref{*j0}) we have
   $$
   \height_{\mathbf{H}}({t})\ =
   \ \height_{\mathbf{H}}({v})+\height_{\mathbf{G}({v})}({t})\ =
   \ \height_{\mathbf{H}}({v})+\height_{\mathbf{G}({v})}({s})-{1}\ =
   $$
   $$
   \height_{\mathbf{H}}({v})+\big(\height\mathbf{G}({v})-{1}\big)-{1}\ =
   \ \height_{\mathbf{H}}({v})+\big({k}({v})+{1}\big)-{2}\ =
   $$
   $$
   \alpha\big(\height_{\mathbf{F}}({v})-{1}\big)+{1}+\alpha\big(\height_{\mathbf{F}}({v})\big)
   -\alpha\big(\height_{\mathbf{F}}({v})-{1}\big)-{1}\ =
   \ \alpha\big(\height_{\mathbf{F}}({v})\big)\ =
   \ \alpha({r}).
   $$
\end{proof}

\section{Nice $\pi\!$-tree for a co-countable subspace}

In this section we prove Corollary~\ref{cor.pi.tree.for.X-A}, which states that if a space ${X}$ has a ``very nice'' $\pi\!$-tree (that is, a $\pi\!$-tree $\mathbf{F}\hspace{-1pt}$ such that
$\cofin\omega\gg\riseee_{\mathbf{F}}({X})$) and if ${A}\subseteq{X}\hspace{-1pt}$ is at most countable, then the subspace ${X}\setminus{A}$ has a ``nice'' $\pi\!$-tree --- that is, a $\pi\!$-tree that satisfies the conditions of Theorem~\ref{th.2}.
This result allows to apply Theorem~\ref{th.2} to co-countable subspaces of the Sorgenfrey line, see (c) of Lemma~\ref{lem.pi.tree.for.N.and.Sorg.l} and Corollary~\ref{cor.rise.in.N.and.sorg.line} in Section~\ref{sect.examples}.

\begin{prop}\label{prop.f.tree.for.X-A}
    Suppose that $\mathbf{F}\hspace{-1pt}$ is a Baire foliage tree on a space ${X}\hspace{-1pt}$ and ${A}\subseteq{X}\hspace{-1pt}$ is at most countable.
    Then there exists a Baire foliage tree $\mathbf{H}\hspace{-1pt}$ on the subspace ${X}\setminus{A}$ such that for every ${p}\in{X}\setminus{A},$ there is a strictly increasing function ${f}_{{p}}\colon\omega\to\omega$ with a property

   \begin{itemize}
   \item [\textup{(\ding{169})}]
      $\mathsurround=0pt
      \big\{\,2{n}\,{+}\,{1}\::\:{n}\,{\in}\,\omega\textup{\ \,\textsf{and}\ }
      {f}_{{p}}({n})\,{\in}\hspace{1pt}\rise_{\mathbf{F}}({p},{U})\,\big\}
      \;\subseteq\;\rise_{\mathbf{H}}({p},{U}{\setminus}\hspace{0.7pt}{A})\quad\ $
      for all ${U}\in\nbhds({p},{X}).$
   \end{itemize}
\end{prop}

\begin{sle}\label{cor.pi.tree.for.X-A}
   Suppose that $\mathbf{F}\hspace{-1pt}$ is a $\pi\!$-tree on a space ${X}\hspace{-1pt}$ such that
   $\cofin\omega\gg\riseee_{\mathbf{F}}({X})$ and ${A}\subseteq{X}\hspace{-1pt}$ is at most countable.
   Then there exists a $\pi\!$-tree $\mathbf{H}\hspace{-1pt}$ on the subspace ${X}\setminus{A}$ such that
   %    Suppose that $\mathbf{F}\hspace{-1pt}$ is a $\pi\!$-tree on a space ${X},$
   %    ${A}\in[{X}]^{\leqslant\omega},$ and $\cofin\omega\gg\riseee_{\mathbf{F}}({X}).$
   %    Then there exists a $\pi\!$-tree $\mathbf{H}$ on the subspace ${X}\setminus{A}$ such that

   \begin{itemize}
   \item [\textup{(\ding{171})}]
      $\mathsurround=0pt
      \cofin\{2{n}+1:{n}\in\omega\}\;\gg\;\riseee_{\mathbf{H}}({X}\setminus{A}).$
   \end{itemize}
\end{sle}

\begin{rem}
   In statements of Proposition~\ref{prop.f.tree.for.X-A} and Corollaries~\ref{cor.pi.tree.for.X-A} the sequence $\langle2{n}+1\rangle_{{n}\in\omega}$ can be replaced by an arbitrary sequence $\langle{k}_{{n}}\rangle_{{n}\in\omega}$ of natural numbers such that  ${k}_{{0}}\geqslant1$ and ${k}_{{n}+1}>{k}_{{n}}+{1}$ for all ${n}\in\omega.$
\end{rem}

\begin{proof}[\textbf{\textup{Proof of Corollary~\ref{cor.pi.tree.for.X-A}}}]
   Let $\mathbf{H}$ be a Baire foliage tree on the subspace ${X}\setminus{A}$ from Proposition~\ref{prop.f.tree.for.X-A}. First we show that condition (\ding{171}) is satisfied. Suppose that ${D}\in\riseee_{\mathbf{H}}({X}\setminus{A});$ that is,
   ${D}=\rise_{\mathbf{H}}({p},{U}\setminus{A})$ for some ${p}\in{X}\setminus{A}$ and ${U}\in\nbhds({p},{X}).$
   Let ${f}_{{p}}\colon\omega\to\omega$ be a function that satisfies condition (\ding{169}) of Proposition~\ref{prop.f.tree.for.X-A}.
   Since $\mathbf{F}\hspace{-1pt}$ is a $\pi\!$-tree on ${X},$ then $\rise_{\mathbf{F}}({p},{U})\neq\varnothing$ by (a) of Lemma~\ref{rem.rise.vs.grow.into}, so it follows from $\cofin\omega\gg\riseee_{\mathbf{F}}({X})$ that
   there is some ${\bar{m}}\in\omega$ such that
   $\omega\setminus{\bar{m}}\subseteq\rise_{\mathbf{F}}({p},{U}).$
   Therefore, since ${f}_{{p}}$ is strictly increasing, there is some ${\bar{n}}\in\omega$ such that ${f}_{{p}}({n})\in\rise_{\mathbf{F}}({p},{U})$ for all ${n}\geqslant{\bar{n}}.$
   Then by (\ding{169}) we have
   $$
   \big\{2{n}+1:{n}\in\omega\setminus{\bar{n}}\big\}
   \ \subseteq\ \rise_{\mathbf{H}}({p},{U}\setminus{A})\ =\ {D},
   $$
   hence (\ding{171}) is satisfied.

   It follows from the above reasoning that ${D}\neq\varnothing,$ so $\varnothing\nin\riseee_{\mathbf{H}}({X}\setminus{A}),$ and hence $\mathbf{H}$ grows into ${X}\setminus{A}$ by (a) of Lemma~\ref{rem.rise.vs.grow.into}.
   This means that $\mathbf{H}$ is a $\pi\!$-tree on ${X}\setminus{A}.$
\end{proof}

In the following lemma we use terminology of the foliage hybrid operation, see Definition~\ref{def.graft} in Section~\ref{f.h.o.}.

\begin{lem}\label{lem.graft.for.X-A}
    Suppose that $\mathbf{F}\hspace{-1pt}$ is a Baire foliage tree on a space ${X},$ ${p}\in{X},$ and ${v}\in\scope_{\mathbf{F}}({p}).$
   Then there exists a foliage tree $\mathbf{G}$ such that

   \begin{itemize}
   \item [\textup{(d1)}]
      $\mathsurround=0pt
      {0}_{\mathbf{G}}={v}\enskip$ and $\enskip\maxel\mathbf{G}=\sons_{\mathbf{G}}({0}_{\mathbf{G}});$
   \item [\textup{(d2)}]
      $\mathsurround=0pt
      \mathbf{G}_{{v}}\,=\:\mathbf{F}_{\hspace{-1.2pt}{v}}\,{\setminus}\,\{{p}\}\:\equiv
      \ \bigsqcup_{{m}\in\maxel\mathbf{G}}\mathbf{F}_{\hspace{-1.2pt}{m}};$
%   \item [\textup{(d5)}]
%      $\mathsurround=0pt \mathbf{F}_{\hspace{-1.2pt}{v}}\setminus\{{p}\}
    \item [\textup{(d3)}]
      $\mathsurround=0pt
      \mathbf{G}\,$ is a foliage graft for $\mathbf{F};$
   \item [\textup{(d4)}]
      $\mathsurround=0pt
      \impl\mathbf{G}=\varnothing;$
   \item [\textup{(d5)}]
      $\mathsurround=0pt
      \mathbf{G}\,$ is $\omega\!$-branching, locally strict, open in ${X},$
      has bounded chains, and $\height\mathbf{G}={2}.$
%   \item [\textup{(d5)}]
%      $\mathsurround=0pt
%      \height\mathbf{G}=2;$
%       \item [\textup{(d8)}]
%      $\mathsurround=0pt
%      \mathbf{G}$ is $\omega\!$-branching,
%   \item [\textup{(d9)}]
%      $\mathsurround=0pt
%      \mathbf{G}$ has bounded chains,
%   \item [\textup{(d10)}]
%      $\mathsurround=0pt
%      \mathbf{G}$ is locally strict, and
%   \item [\textup{(d11)}]
%      $\mathsurround=0pt
%      \mathbf{G}$ is open in ${X}.$
   \end{itemize}
\end{lem}

\begin{proof}[\textbf{\textup{Proof of Lemma~\ref{lem.graft.for.X-A}}}]
   Put
   $$
   {B}\;\coloneq\;{\bigcup}\big\{\sons_{\mathbf{F}}({u}):
   {u}\in\scope_{\mathbf{F}}({p})\,{\cap}\,{v}{\downfilledspoon}_{\hspace{-0.5pt}\mathbf{F}}\big\}
   \qquad\text{and}\qquad
   \mathsf{MAX}\;\coloneq\;{B}\setminus\scope_{\mathbf{F}}({p}).
   $$
   Let $\mathcal{T}\hspace{-1pt}$ be a partial order such that
   $$\nodes\mathcal{T}\;\coloneq\;\{{v}\}\cup\mathsf{MAX}
   \qquad\text{and}\qquad
   {<}_{\mathcal{T}}\;\coloneq\;\big\{({v},{m}):{m}\in\mathsf{MAX}\big\}.
   $$
   Then $\mathcal{T}\hspace{-1pt}$ is a tree, ${0}_\mathcal{T}={v},$ $\maxel\mathcal{T}=\mathsf{MAX},$ and $\mathcal{T}\hspace{-1pt}$ is a graft for $\skeleton\mathbf{F}.$
   Now let $\mathbf{G}$ be a foliage tree with $\skeleton\mathbf{G}\coloneq\mathcal{T}\hspace{-1pt}$ and with leaves $\mathbf{G}_{{v}}\coloneq\mathbf{F}_{\hspace{-1.2pt}{v}}\setminus\{{p}\}$ and $\mathbf{G}_{{m}}\coloneq\mathbf{F}_{\hspace{-1.2pt}{m}}$ for all ${m}\in\maxel\mathcal{T}.$ Then using (b) of Lemma~\ref{lem.F.vs.S.by.isomorph} and Corollary~\ref{cor.scope.in.Baire.f.tree} it is not hard to verify that clauses (d1)--(d5) are satisfied.
\end{proof}

\begin{proof}[\textbf{\textup{Proof of Proposition~\ref{prop.f.tree.for.X-A}}}]
   We may assume that ${A}=\big\{{p}_{{i}}:{i}\in{|}{A}{|}\big\}$ and ${p}_{{i}}\neq{p}_{{j}}$ for all ${i}\neq{j}\in{|}{A}{|}.$ First we build sequences $\langle{M}_{{i}}\rangle_{{i}\in{|}{A}{|}},$ $\langle{z}_{{i}}\rangle_{{i}\in{|}{A}{|}},$ and $\big\langle\mathbf{G}({i})\big\rangle_{{i}\in{|}{A}{|}}$ by recursion on ${i}\in{|}{A}{|}$:

   \begin{itemize}
   \item [\textup{(e1)}]
      $\mathsurround=0pt
      {M}_{{0}}\,\coloneq\,\{{0}_{\mathbf{F}}\};$
   \item [\textup{(e2)}]
      $\mathsurround=0pt
      {z}_{i}\;\coloneq\;$ the node in $\mathbf{F}\hspace{-1pt}$ such that $\{{z}_{{i}}\}={M}_{{i}}\cap\scope_{\mathbf{F}}({p}_{{i}});$
   \item [\textup{(e3)}]
      $\mathsurround=0pt
      \mathbf{G}({i})\;\coloneq\;$ the foliage tree $\mathbf{G}$ from Lemma~\ref{lem.graft.for.X-A} with ${p}={p}_{{i}}$ and ${v}={z}_{{i}};$
   \item [\textup{(e4)}]
      $\mathsurround=0pt
      {M}_{{i}+1}\;\coloneq\;\big({M}_{{i}}\,{\setminus}\,\{{z}_{i}\}\big)\,\cup\,
      {\displaystyle\bigcup}\big\{\sons_{\mathbf{F}}({m}):{m}\,{\in}\hspace{1pt}\maxel\mathbf{G}({i})\big\}.$
   \end{itemize}
The correctness of clause (e2) follows from (f3), see below.

It is not hard to verify that for each ${i}\in{|}{A}{|},$ the following conditions are satisfied:

   \begin{itemize}
   \item [\textup{(f1)}]
      $\mathsurround=0pt
      {M}_{{i}}\,$ is an antichain in $\mathbf{F}\hspace{-1pt}$
      (that is, ${u}\nless_\mathbf{F} {v}\enskip\mathsf{and}\enskip{v}\nless_\mathbf{F}{u}$
      for all ${u},{v}\in{M}_{{i}}$);
   \item [\textup{(f2)}]
      $\mathsurround=0pt
      {M}_{{i}+{1}}\subseteq({M}_{{i}}){\downfootline}_{\hspace{-0.5pt}\mathbf{F}};$
   \item [\textup{(f3)}]
      $\mathsurround=0pt
      {X}\setminus\{{p}_{{j}}:{j}\in{i}\}\;\equiv
      \ \bigsqcup_{{m}\in{M}_{{i}}}\mathbf{F}_{\hspace{-1.2pt}{m}};$
   \item [\textup{(f4)}]
      $\mathsurround=0pt
      {0}_{\mathbf{G}({i})}\in{M}_{{i}};$
   \item [\textup{(f5)}]
      $\mathsurround=0pt
      \cut\big(\mathbf{F},\mathbf{G}({i})\big)=\{{p}_{{i}}\};$
   \item [\textup{(f6)}]
      $\mathsurround=0pt
      \big\{\mathbf{G}({j}):{j}\in{i}+1\big\}\,$ is a consistent family of foliage grafts for $\mathbf{F}.$
   \end{itemize}

   Now it follows from (f6) that $\varphi\coloneq\big\{\mathbf{G}({i}):{i}\in{|}A{|}\big\}$ is a consistent family of foliage grafts for $\mathbf{F},$ so we can define $\mathbf{H}\coloneq\fhybr(\mathbf{F},\varphi).$
   Then $\mathbf{H}$ is a Baire foliage tree on ${X}\setminus{A}$ by Lemma~\ref{lem.when.f.hybr.is.pi.tree}, (f5), and (d5). Also $\mathbf{H}$ satisfies the following:

   \begin{itemize}
   \item [\textup{(g1)}]
      $\mathsurround=0pt
      \mathbf{H}\,$ has nonempty leaves.

      This follows from (b) of Corollary~\ref{cor.scope.in.Baire.f.tree}.
   \item [\textup{(g2)}]
      $\mathsurround=0pt
      \nodes\mathbf{H}\subseteq\nodes\mathbf{F}.$

      This follows from (d4).
   \item [\textup{(g3)}]
      $\mathsurround=0pt
      \mathbf{H}_{{u}}=\mathbf{F}_{\hspace{-1.2pt}{u}}\setminus{A}
      \quad$ for all ${u}\in\nodes\mathbf{H}.$

      This also follows from (d4).
   \item [\textup{(g4)}]
      $\mathsurround=0pt
      \height_{\mathbf{H}}({u})\in\{2{n}:{n}\in\omega\}\quad$
      for all ${u}\in{M}_{{i}}$ and ${i}\in{|}{A}{|}.$

      We prove (g4) by induction on ${i}\in{|}{A}{|}.$ Obviously, $\height_{\mathbf{H}}({u})$ is even when ${u}\in{M}_{{0}}.$ Assume as inductive hypothesis that the assertion of (g4) holds for all ${u}\in\bigcup_{{i}\leqslant{k}}{M}_{{i}}$ and prove it for an arbitrary ${u}\in{M}_{{{k}+{1}}}.$
      If ${u}\in{M}_{{k}},$
      then $\height_{\mathbf{H}}({u})$ is even by the inductive hypothesis.
      If ${u}\nin{M}_{{k}},$ then (e4) implies that ${u}\in\sons_{\mathbf{F}}({m})$
      for some ${m}\in\maxel\mathbf{G}({k}).$
      We have
      $$
      \maxel\mathbf{G}({k})=\sons_{\mathbf{G}({k})}({0}_{\mathbf{G}({k})})
      \qquad\text{and}\qquad
      \sons_{\mathbf{G}({k})}({0}_{\mathbf{G}({k})})=
      \sons_{\mathbf{H}}({0}_{\mathbf{G}({k})})
      $$
      by (d1) and by (a) of Proposition~\ref{prop.hybr}, so it follows from (f4) and from the inductive hypothesis that $\height_{\mathbf{H}}({m})$ is odd
      and then, using (f4) and the inductive hypothesis again, we have
      ${m}\nin\big\{{0}_{\mathbf{G}({j})}:{j}\in{k}+{1}\big\}.$
      %${m}\neq{0}_{\mathbf{G}({j})}$ for all ${j}\in{k}+{1}.$
      Let us show that
      ${m}\nin\big\{{0}_{\mathbf{G}({j})}:{j}\in{|}{A}{|}\big\}.$
      If not, then
      $$
      {m}\,\in\,\big\{{0}_{\mathbf{G}({j})}:{j}\in{|}{A}{|}\setminus({k}+{1})\big\},
      \qquad\text{so}\qquad
      {m}\,\in\:\bigcup\big\{{M}_{{j}}:{j}\in{|}{A}{|}\setminus({k}+{1})\big\}
      $$
      by (f4).
      Then it follows from (f2) that ${m}\in({M}_{{k}+{1}}){\downfootline}_{\hspace{-0.5pt}\mathbf{F}}$
      --- that is, ${m}\geqslant_{\mathbf{F}}{t}$ for some ${t}\in{M}_{{k}+{1}}.$
      Since ${u}\in\sons_{\mathbf{F}}({m}),$ then ${u}>_{\mathbf{F}}{m},$ therefore ${u}>_{\mathbf{F}}{t}.$
      This contradicts (f1) because ${t},{u}\in{M}_{{k}+{1}}.$
      Now it follows from (d4) and (a) of Proposition~\ref{prop.hybr} that
      $\sons_{\mathbf{F}}({m})=\sons_{\mathbf{H}}({m}),$
      so ${u}\in\sons_{\mathbf{H}}({m}).$
      Then $\height_{\mathbf{H}}({u})\in\{{2}{n}:{n}\in\omega\}$
      because $\height_{\mathbf{H}}({m})$ is odd.
   \end{itemize}

   Now suppose that ${p}\in{X}\setminus{A};$ we must find a strictly increasing function ${f}_{{p}}\colon\omega\to\omega$ that satisfies (\ding{169}). First, using (d) of Corollary~\ref{cor.scope.in.Baire.f.tree}, for each ${n}\in\omega,$ we define ${m}({p},{n})$ to be the node in $\scope_{\mathbf{H}}({p})$ such that
   $\height_{\mathbf{H}}\big({m}({p},{n})\big)=2{n}+1.$
   Using (g3) we have
   \begin{equation}\label{*e1}
      \forall{n}\,{\in}\,\omega\ \big[{m}({p},{n})\in\scope_{\mathbf{F}}({p})\big].
   \end{equation}
   Next, using (g2), we can define $$
   {f}_{{p}}({n})\coloneq\height_{\mathbf{F}}\big({m}({p},{n})\big)
   \quad\text{ for every }\;{n}\in\omega;
   $$
   then ${f}_{{p}}\colon\omega\to\omega$ by (a) of Corollary~\ref{cor.scope.in.Baire.f.tree}.
   If ${n}''>{n}',$ then ${m}({p},{n}'')>_{\mathbf{H}}{m}({p},{n}')$ (because $\scope_{\mathbf{H}}({p})$ is a chain in $\mathbf{H}$ by (c) of Corollary~\ref{cor.scope.in.Baire.f.tree}), so ${m}({p},{n}'')>_{\mathbf{F}}{m}({p},{n}')$ by (g2) and (b) of Lemma~\ref{prop.hybr}. This implies that ${f}_{{p}}$ is strictly increasing.
   Now, using (f4) and (g4), for every ${n}\in\omega,$ we have $$
   {m}({p},{n})\nin\big\{{0}_{\mathbf{G}({i})}:{i}\in{|}{A}{|}\big\},
   \quad\text{ so }\quad
   \sons_{\mathbf{H}}\big({m}({p},{n})\big)=\sons_{\mathbf{F}}\big({m}({p},{n})\big)
   $$
   by (d4) and (a) of Proposition~\ref{prop.hybr}. Then (g3) and (g1) imply
   \begin{equation}\label{*e4}
      \shoot_{\mathbf{H}}\big({m}({p},{n})\big)\gg\shoot_{\mathbf{F}}\big({m}({p},{n})\big)
      \qquad\text{for all}\enskip{n}\in\omega.
   \end{equation}

   To complete the proof it remains to verify (\ding{169}); suppose that $$
   {U}\in\nbhds({p},{X}),\quad{n}\in\omega,
   \quad\text{and}\quad{f}_{{p}}({n})\in\rise_{\mathbf{F}}({p},{U}).
   $$
   The last formula means that ${f}_{{p}}({n})=\height_{\mathbf{F}}({v})$ for some ${v}\in\scope_{\mathbf{F}}({p})$ such that
   $\shoot_{\mathbf{F}}({v})\gg\{{U}\}.$
   Then ${v}={m}({p},{n})$ by (d) of Corollary~\ref{cor.scope.in.Baire.f.tree}, by (\ref{*e1}), and by definition of  ${f}_{{p}}({n}).$ It follows that $$
   \shoot_{\mathbf{F}}\big({m}({p},{n})\big)\gg\{{U}\},
   \quad\text{ so }\quad
   \shoot_{\mathbf{H}}\big({m}({p},{n})\big)\gg\{{U}\}
   $$
   by (\ref{*e4}). Then (g3) implies $\shoot_{\mathbf{H}}\big({m}({p},{n})\big)\gg\{{U}\setminus{A}\},$ therefore
   $$
   2{n}+1\;=\;\height_{\mathbf{H}}\big({m}({p},{n})\big)
   \;\in\;\rise_{\mathbf{H}}({p},{U}\setminus{A})
   $$
   by definition of $\rise_{\mathbf{H}}({p},{U}\setminus{A}).$
\end{proof}

\section{New examples of spaces with a $\pi\!$-tree}
\label{sect.examples}

Recall that $\mathcal{N}$ is the Baire space, $\mathcal{R}_{\scriptscriptstyle\mathcal{S}}$ is the Sorgenfrey line, and $\mathcal{I}_{\scriptscriptstyle\mathcal{S}}\coloneq
\mathcal{R}_{\scriptscriptstyle\mathcal{S}}\setminus\mathbb{Q}$ is the irrational Sorgenfrey line.
%Note that no finite power of the Sorgenfrey line is homeomorphic to finite power of the irrational Sorgenfrey line~\cite{Pacific}.

\begin{lem}\label{lem.pi.tree.for.N.and.Sorg.l}\mbox{ }

   \begin{itemize}
   \item[\textup{(a)}]
      $\mathsurround=0pt
      \mathcal{N}\,$ has a $\pi\!$-tree $\mathbf{F}\hspace{-1pt}$ such that $\cofin\omega\gg\riseee_{\mathbf{F}}(\mathcal{N}).$
   \item[\textup{(b)}]
      $\mathsurround=0pt
      \mathcal{R}_{\scriptscriptstyle\mathcal{S}}\,$ has a $\pi\!$-tree $\mathbf{G}$ such that $\cofin\omega\gg\riseee_{\mathbf{G}}(\mathcal{R}_{\scriptscriptstyle\mathcal{S}}).$
   \item[\textup{(c)}]
      If ${X}\subseteq\mathcal{R}_{\scriptscriptstyle\mathcal{S}}$ and $\mathcal{R}_{\scriptscriptstyle\mathcal{S}}\setminus{X}\hspace{-1pt}$ is at most countable,

      then ${X}$ has a $\pi\!$-tree $\mathbf{H}\hspace{-1pt}$ such that $\cofin\{2{n}\,{+}\,1:{n}\in\omega\}\gg\riseee_{\mathbf{H}}({X}).$
   \end{itemize}
\end{lem}

\begin{proof}
   Part (a) follows from (b) of Lemma~\ref{lem.pi.and.B.f.trees.vs.S} and Example~\ref{exmpl.rise.in.S}; part (b) can be derived from the proof of Lemma~3.6 in~\cite{MPatr}; part (c) follows from part (b) and Corollary~\ref{cor.pi.tree.for.X-A}.
\end{proof}

Using the above lemma and Theorem~\ref{th.2} we obtain the following statement:

\begin{sle}\label{cor.rise.in.N.and.sorg.line}
   Suppose that ${1}\leqslant{|}{A}{|}\leqslant\omega$ and for each  $\alpha\in{A},$
   $$
   \text{either}\qquad{X}_{\alpha}=\mathcal{N}
   \qquad\text{or}\qquad
   {X}_{\alpha}\subseteq\mathcal{R}_{\scriptscriptstyle\mathcal{S}}
   \ \text{ with }\ {|}\mathcal{R}_{\scriptscriptstyle\mathcal{S}}\setminus{X}_{\alpha}{|}\leqslant\omega.
   $$
   Then the product $\prod_{\alpha\in{A}}{X}_{\alpha}$ has a $\pi\!$-tree.
   \hfill$\qed$
\end{sle}

\begin{sle}\label{cor.pi.tree.for.powers.of.(irr)Sor.l}\mbox{ }

   \begin{itemize}
   \item[\textup{(a)}]
      $\mathsurround=0pt
      {\mathcal{R}_{\scriptscriptstyle\mathcal{S}}}^{{n}}\hspace{-1pt}$ and $\hspace{1pt}{\mathcal{I}_{\scriptscriptstyle\mathcal{S}}}^{{n}}\hspace{-1pt}$ have a $\pi\!$-tree for all ${n}\in\omega\setminus\{{0}\}.$
   \item[\textup{(b)}]
      $\mathsurround=0pt
      {\mathcal{R}_{\scriptscriptstyle\mathcal{S}}}^{\omega}\hspace{-1pt}$ and $\hspace{1pt}{\mathcal{I}_{\scriptscriptstyle\mathcal{S}}}^{\omega}\hspace{-1pt}$ have a $\pi\!$-tree.
      \hfill$\qed$
   \end{itemize}
\end{sle}
\noindent
Note that if ${X}\subseteq\mathcal{N}$ with
${|}\hspace{1pt}\mathcal{N}\setminus{X}{|}\leqslant\omega,$
then ${X}\hspace{-1pt}$ is homeomorphic to $\mathcal{N}$ (this can be easily derived from the Alexandrov-Urysohn characterization of the Baire space and from the characterization of its Polish subspaces --- see Theo\-rems~3.11 and~7.7 in~\cite{kech}).
Notice also that $\mathcal{N}^{{n}}$ is homeomorphic to $\mathcal{N}$ for all ${n}\in\omega\setminus\{{0}\}$ and $\mathcal{N}^\omega$ is also homeomorphic to $\mathcal{N}.$

%Note that if ${X}\subseteq\mathcal{N}$ with
%${|}\hspace{1pt}\mathcal{N}\setminus{X}{|}\leqslant\omega,$
%then ${X}\hspace{-1pt}$ is homeomorphic to $\mathcal{N}$ (this can be easily derived from the Alexandrov-Urysohn characterization of the Baire space and from the characterization of its Polish subspaces --- see Theo\-rems~3.11 and~7.7 in~\cite{kech}).
%Notice also that any at most countable power of $\mathcal{N}$ is homeomorphic to $\mathcal{N}.$
%
%Note that if ${X}\subseteq\mathcal{N}$ with
%${|}\hspace{1pt}\mathcal{N}\setminus{X}{|}\leqslant\omega,$
%then ${X}\hspace{-1pt}$ is homeomorphic to $\mathcal{N}$ (this can be easily derived from the Alexandrov-Urysohn characterization of the Baire space and from the characterization of its Polish subspaces --- see Theo\-rems~3.11 and~7.7 in~\cite{kech}).
%Notice also that $\mathcal{N}^\omega$ is homeomorphic to $\mathcal{N}$ and the powers $\mathcal{N}^{{n}}$ are homeomorphic to $\mathcal{N}$ for all ${n}\in\omega\setminus\{{0}\}.$

\begin{sle}\label{cor.pi.tree.for.X*N.X*Sor.l}
   If a space ${X}$ has a $\pi\!$-tree,
   then ${X}\times\mathcal{N},$ ${X}\times\mathcal{R}_{\scriptscriptstyle\mathcal{S}},$ and ${X}\times{\mathcal{R}_{\scriptscriptstyle\mathcal{S}}}^{\omega}\hspace{-1.5pt}$ also have a $\pi\!$-tree.
\end{sle}

\begin{proof} This statement follows from Corollary~\ref{cor.Y*product} and Lemma~\ref{lem.pi.tree.for.N.and.Sorg.l}.
\end{proof}

\section{Appendix. The foliage hybrid operation}
\label{f.h.o.}

In the proofs of Lemma~\ref{lem.perestrojka.LPB.po.sdvigu.fil'tra} and Proposition~\ref{prop.f.tree.for.X-A} we employ the foliage hybrid operation, which was introduced in \cite{my.paper}. For completeness of exposition we list here definitions and results that we use.
The definition of graft, which we give below, slightly differs from the definition of graft in \cite{my.paper}, but these two definitions are easily seen to be equivalent. The same can be said about our definition of $\hybr(\mathcal{T}\hspace{-1pt},\gamma),$ see details in \cite[Remark~20]{my.paper}.
To ease comprehension of notions from Definition~\ref{def.graft}, you can look at pictures that illustrate this definition in~\cite{patr.slides}.

\begin{notation}\mbox{ }

   \begin{itemize}
   \item [\ding{46}]
      $\mathsurround=0pt
      \forall{x}\,{\neq}\,{y}\,{\in}\,{A}\ \varphi({x},{y})
      \quad{\colon}{\longleftrightarrow}\quad
      \forall{x},{y}\,{\in}\,{A}\ \big[{x}\neq {y}\to\varphi({x},{y})\big];$
   \item [\ding{46}]
      $\mathsurround=0pt
      {x}\parallel_{\hspace{-1pt}\mathcal{P}} {y}
      \quad{\colon}{\longleftrightarrow}\quad
      {x}\nleqslant_\mathcal{P} {y}\enskip\mathsf{and}\enskip{x}\ngtr_\mathcal{P}{y}.$
   \end{itemize}
\end{notation}

\begin{deff}[{Definitions~15, 17, 19, 25--27 and Remark~20 in \cite{my.paper}}]\label{def.graft}
   Suppose that
   $\mathcal{T},\hspace{1pt}\mathcal{G}$ are trees and $\mathbf{F},\mathbf{G}$ are nonincreasing foliage trees.

   \begin{itemize}
   \item [\ding{46}]
      $\mathsurround=0pt
      \mathcal{G}\,$ is a \textbf{graft} for $\mathcal{T}
      \quad{\colon}{\longleftrightarrow}\quad$
      \begin{itemize}
      \item[\ding{226}]
         $\mathsurround=0pt
         {|}\nodes\mathcal{G}{|}>1,$
      \item[\ding{226}]
         $\mathsurround=0pt
         \mathcal{G}\,$ has the least node,
      \item[\ding{226}]
         $\mathsurround=0pt
         \nodes\mathcal{G}\cap\nodes\mathcal{T}\:=\;
         \{{0}_\mathcal{G}\}\cup\maxel\mathcal{G},$\quad and
      \item[\ding{226}]
         $\mathsurround=0pt
         \forall{x},{y}\in\nodes\mathcal{G}\cap\nodes\mathcal{T}
         \ [{x}<_\mathcal{G} {y}\leftrightarrow{x}<_{\mathcal{T}} {y}].$
      \end{itemize}
   \item [\ding{46}]
      If $\mathcal{G}$ is a graft for $\mathcal{T}\hspace{-1pt},$ then\textup{:}
      \begin{itemize}
      \item[\ding{226}]
         $\mathsurround=0pt
         \impl\mathcal{G}\;\coloneq\;\nodes\mathcal{G}\setminus
         \big(\{{0}_\mathcal{G}\}\cup\maxel\mathcal{G}\big);$
      \item[\ding{226}]
         $\mathsurround=0pt
         \expl(\mathcal{T}\hspace{-1pt},\mathcal{G})\;\coloneq\; ({0}_\mathcal{G}){\downspoon}_{\hspace{-0.5pt}\mathcal{T}}
         \setminus(\maxel\mathcal{G}){\downfootline}_{\hspace{-0.5pt}\mathcal{T}}.$
         $\vphantom{\bar{[]}}$
      \end{itemize}
   \item [\ding{46}]
      $\mathsurround=0pt
      \gamma\,$ is a \textbf{consistent} family of grafts for $\mathcal{T}
      \quad{\colon}{\longleftrightarrow}\quad$
      \begin{itemize}
      \item[\ding{226}]
         $\mathsurround=0pt
         \forall\mathcal{G}\,{\in}\,\gamma
         \ [\:\mathcal{G}$ is a graft for $\mathcal{T}\:],$
      \item[\ding{226}]
         $\mathsurround=0pt
         \forall\mathcal{D}\,{\neq}\,\mathcal{E}\,{\in}\,\gamma
         \ [\:\impl\mathcal{D}\cap\impl\mathcal{E}=\varnothing\:],$\enskip and
         $\vphantom{\overline{[]}}$
      \item[\ding{226}]
         $\mathsurround=0pt
         \forall\mathcal{D}\,{\neq}\,\mathcal{E}\,{\in}\,\gamma
         \ \big[\
         {0}_\mathcal{D}\parallel_{\hspace{-1pt}\mathcal{T}}{0}_\mathcal{E}
         \enskip\textsf{ or }\enskip
         {0}_\mathcal{D}\in(\maxel\mathcal{E}) {\downfootline}_{\hspace{-0.5pt}\mathcal{T}}
         \enskip\textsf{ or }\enskip
         {0}_\mathcal{E}\in(\maxel\mathcal{D}){\downfootline}_{\hspace{-0.5pt}\mathcal{T}}
         \ \big].$
         $\vphantom{\widetilde{[]}}$
      \end{itemize}
   \item [\ding{46}]
      If $\gamma$ is a consistent family of grafts for $\mathcal{T}\hspace{-1pt},$ then\textup{:}
      \begin{itemize}
      \item[\ding{226}]
         $\displaystyle\mathsurround=0pt
         \supp(\mathcal{T}\hspace{-1pt},\gamma)\;\coloneq\;
         \nodes\mathcal{T}\setminus\bigcup_{\mathcal{G}\in\gamma}\expl(\mathcal{T}\hspace{-1pt},\mathcal{G});
         \vphantom{\hat{\bigcup}}$
         \vspace{-1pt}
      \item[\ding{226}]
         $\mathsurround=0pt
         \hybr(\mathcal{T}\hspace{-1pt},\gamma)\;\coloneq\;$
         the pair $({H},{<})$ (actually, a tree) such that
         \begin{itemize}
         \item[\ding{51}]
            $\displaystyle\mathsurround=0pt
            \ {H}\;\coloneq\ \supp(\mathcal{T}\hspace{-1pt},\gamma)\cup
            \bigcup_{\mathcal{G}\in\gamma}\impl\mathcal{G}\enskip
            \vphantom{\bar{\hat{\bigcup}}}$ and
            \vspace{-2pt}
         \item[\ding{51}]
            $\mathsurround=0pt
            \ \,{<}\ \coloneq\ $ the transitive closure of relation
            $\displaystyle({<}_{\mathcal{T}}\cup\bigcup_{\mathcal{G}\in\gamma}\!{<}_\mathcal{G})
            \,\cap\,({H}\times {H}).$
            \vspace{-3pt}
         \end{itemize}
      \end{itemize}
   \item [\ding{46}]
      $\mathsurround=0pt
      \mathbf{G}\,$ is a \textbf{foliage graft} for $\mathbf{F}
      \quad{\colon}{\longleftrightarrow}\quad$
      \begin{itemize}
      \item[\ding{226}]
         $\mathsurround=0pt
         \mathbf{G}\,$ is nonincreasing,
      \item[\ding{226}]
         $\mathsurround=0pt
         \skeleton\mathbf{G}\,$ is a graft for $\skeleton\mathbf{F},$
      \item[\ding{226}]
         $\mathsurround=0pt
         \mathbf{G}_{{0}_{\mathbf{G}}}
         \subseteq\mathbf{F}_{\hspace{-1.2pt}{0}_{\mathbf{G}}},\enskip$ and
      \item[\ding{226}]
         $\mathsurround=0pt
         \forall{m}\in\maxel\mathbf{G}
         \ [\mathbf{G}_{{m}}=\mathbf{F}_{\hspace{-1.2pt}{m}}
         ].$
      \end{itemize}
   \item [\ding{46}]
      If $\mathbf{G}$ is a foliage graft for $\mathbf{F},$ then
      \begin{itemize}
      \item[\ding{226}]
         $\mathsurround=0pt
         \cut(\mathbf{F},\mathbf{G})\coloneq   \mathbf{F}_{\hspace{-1.2pt}{0}_{\mathbf{G}}}\setminus   \mathbf{G}_{{0}_{\mathbf{G}}}.$
      \end{itemize}
   \item [\ding{46}]
      $\mathsurround=0pt
      \varphi\,$ is a \textbf{consistent} family of foliage grafts for $\mathbf{F}
      \quad{\colon}{\longleftrightarrow}\quad$
      \begin{itemize}
      \item[\ding{226}]
         $\mathsurround=0pt
         \forall\mathbf{G}\,{\in}\,\varphi
         \ [\mathbf{G}$ is a foliage graft for $\mathbf{F}],$
      \item[\ding{226}]
         $\mathsurround=0pt
         \forall\mathbf{D}\,{\neq}\,\mathbf{E}\,{\in}\,\varphi
         \ [\skeleton\mathbf{D}\neq\skeleton\mathbf{E}],$\enskip and
         $\vphantom{\overline{[]}}$
      \item[\ding{226}]
         $\mathsurround=0pt
         \{\skeleton\mathbf{G}:\mathbf{G}\in\varphi\}$ is a consistent family of grafts for $\skeleton\mathbf{F}.$
         $\vphantom{\overline{[]}}$
      \end{itemize}
   \item [\ding{46}]
      If $\varphi$ is a consistent family of foliage grafts for $\mathbf{F},$ then\textup{:}
      \begin{itemize}
      \item[\ding{226}]
         $\displaystyle\mathsurround=0pt
         \loss(\mathbf{F},\varphi)\;\coloneq\;\bigcup_{\mathbf{G}\in\varphi}
         \cut(\mathbf{F},\mathbf{G});\vphantom{\hat{\bigcup}}$
         \vspace{-2pt}
      \item[\ding{226}]
         $\mathsurround=0pt
         \fhybr(\mathbf{F},\varphi)$ $\;\coloneq\;$
         the \textbf{foliage hybrid} of $\mathbf{F}\hspace{-1pt}$ and $\varphi$ $\;\coloneq\;$
         the foliage tree $\mathbf{H}\hspace{-1pt}$ such that
         \begin{itemize}
         \item[\ding{51}]
            $\mathsurround=0pt
            \ \skeleton\mathbf{H}\;\coloneq\;
            \hybr\big(\skeleton\mathbf{F},\{\skeleton\mathbf{G}:\mathbf{G}\in\varphi\}\big)
            \enskip\vphantom{\overline{\hat{\bigcup}}}$
            and
         \item[\ding{51}]
            $\mathsurround=0pt
            \;\mathbf{H}_{{x}}\;\coloneq\;
            \begin{cases}
               \mathbf{G}_{{x}}\setminus\loss(\mathbf{F},\varphi),
               &\text
               {if$\enskip{x}\in\impl\mathbf{G}\enskip$for some$\enskip\mathbf{G}\in\varphi;$}
            \\
               \mathbf{F}_{\hspace{-1.2pt}{x}}\setminus\loss(\mathbf{F},\varphi)
               \vphantom{\widetilde{\big\langle\rangle}},
               &\text
               {otherwise.}
            \end{cases}
            $
         \end{itemize}
      \end{itemize}
   \end{itemize}
\end{deff}

\begin{lem}[{Lemma~21 and Proposition~23 in \cite{my.paper}}]\label{prop.hybr}
   Suppose that $\gamma$ is a consistent family of grafts for a tree $\mathcal{T}\hspace{-1pt},$
   $\mathcal{H}=\hybr(\mathcal{T}\hspace{-1pt},\gamma),$ and $\mathcal{G}\in\gamma.$
   %   Suppose that $\gamma$ is a consistent family of grafts for a tree $\mathcal{T}\hspace{-1pt}$
   %   and $\mathcal{H}=\hybr(\mathcal{T}\hspace{-1pt},\gamma).$

   \begin{itemize}
   \item [\textup{(a)}]
      $\mathsurround=0pt \nodes\mathcal{G}\subseteq\nodes\mathcal{H}\enskip$ and $\enskip
      \forall{x},{y}\,{\in}\nodes\mathcal{G}
      \ [{x}<_\mathcal{H} {y}\leftrightarrow{x}<_{\mathcal{G}}{y}].$
   \item [\textup{(b)}]
      $\mathsurround=0pt
      \supp(\mathcal{T}\hspace{-1pt},\gamma)=\nodes\mathcal{H}\cap\nodes\mathcal{T}
      \enskip$ and $\enskip\forall{x},{y}\,{\in}\,\supp(\mathcal{T}\hspace{-1pt},\gamma)
      \ [{x}<_{\mathcal{H}} {y}\leftrightarrow{x}<_{\mathcal{T}} {y}
      ].$
   \item [\textup{(c)}]
      For each ${x}\in\nodes\mathcal{H},$

      $\mathsurround=0pt
      \sons_{\mathcal{H}}({x})=
      \begin{cases}
         \sons_{\mathcal{G}}({x}),
         &\text
         {if$\enskip{x}\in\{{0}_\mathcal{G}\}\cup\impl\mathcal{G}\enskip$for some$\enskip\mathcal{G}\in\gamma;$
          }
      \\
         \sons_{\mathcal{T}}({x}),
         &\text
         {otherwise \textup{\big(}i.e., when$\enskip{x}\in\supp(\mathcal{T}\hspace{-1pt},\gamma)
         \setminus
         \{{0}_\mathcal{G}:\mathcal{G}\in\gamma\}\textup{\big)}.\vphantom{\widetilde{\big\langle\rangle}}$}
      \end{cases}
      $
   \end{itemize}
\end{lem}

\begin{lem}[{Lemma~30 in \cite{my.paper}}]\label{lem.when.f.hybr.is.pi.tree}
    Suppose that $\mathbf{F}\hspace{-1pt}$ is a Baire foliage tree on a space ${X}\hspace{-1pt}$ and $\varphi$ is a consistent family of foliage grafts for $\mathbf{F}\hspace{-1pt}$ such that every $\mathbf{G}$ in $\varphi$ is $\omega\!$-branching, locally strict, open in ${X},$ has bounded chains, and has $\height\mathbf{G}\leqslant\omega.$
    Then the foliage hybrid of $\mathbf{F}\hspace{-1pt}$ and $\varphi$ is a Baire foliage tree on ${X}\setminus\loss(\mathbf{F},\varphi).$
\end{lem}

\begin{deff}[{Definitions~31, 33 in \cite{my.paper}}]\label{def.shoots.into}
   Suppose that
   $\mathbf{H},\mathbf{F}\hspace{-1pt}$ are nonincreasing foliage trees and $\mathbf{G}$ is a foliage graft for $\mathbf{F}.$

   \begin{itemize}
   \item [\ding{46}]
      $\mathsurround=0pt
      \mathbf{H}\,$ \textbf{shoots into} $\mathbf{F}
      \quad{\colon}{\longleftrightarrow}\quad
      \forall{p}\,{\in}\hspace{1pt}\flesh\mathbf{H}
      \ \,\forall{y}\,{\in}\hspace{1pt}\scope_\mathbf{F}({p})
      \ \,\exists{x}\,{\in}\hspace{1pt}\scope_\mathbf{H}({p})
      \ \,\big[\shoot_\mathbf{H}({x})\gg\shoot_\mathbf{F}({y})\big].$
   \item [\ding{46}]
      $\mathsurround=0pt
      \mathbf{G}\,$ \textbf{preserves shoots} of $\mathbf{F}
      \quad{\colon}{\longleftrightarrow}\quad$
      \begin{itemize}

      \item[]
         for each ${p}\in\flesh\mathbf{G}$ and
         for each ${y}\in\scope_\mathbf{F}({p})
         \cap\big(\{{0}_{\mathbf{G}}\}\cup\expl(\mathbf{F},\mathbf{G})\big)$
      \item[]
         there is ${x}\in\scope_\mathbf{G}({p})
         \cap\big(\{{0}_{\mathbf{G}}\}\cup\impl\mathbf{G}\big)$
         such that
         $\:\big[
         \shoot_{\mathbf{G}}({x})
            \gg\shoot_\mathbf{F}({y})
         \big].$
         $\vphantom{\bar{\big[\big]}}$
      \end{itemize}
   \end{itemize}
\end{deff}

\begin{lem}[{Lemma~34 in \cite{my.paper}}]\label{lem.when.f.hyb.sprts.thrgh}
   Suppose that $\mathbf{F}\hspace{-1pt}$ is a nonincreasing foliage tree,
   $\varphi$ is a consistent family of foliage grafts for $\mathbf{F},$
   the foliage hybrid of $\mathbf{F}\hspace{-1pt}$ and $\varphi$ has nonempty leaves, and
   each $\mathbf{G}\in\varphi$ preserves shoots of $\mathbf{F}.$
   Then the foliage hybrid of $\mathbf{F}\hspace{-1pt}$ and $\varphi$ shoots into $\mathbf{F}.$
\end{lem}

\begin{lem}[{Lemma~32 in \cite{my.paper}}]\label{l.grows.into.subspace}
   Suppose that a foliage tree $\mathbf{H}\hspace{-1pt}$ shoots into a foliage tree~$\mathbf{F}\hspace{-1pt}$
   and $\mathbf{F}\hspace{-1pt}$ grows into a space~${X}.$
   Then
   $\mathbf{H}\hspace{-1pt}$ grows into the subspace ${X}\cap\flesh\mathbf{H}\hspace{-1pt}$ of ${X}.$
\end{lem}

%=================������ ����������====================


\begin{thebibliography}{99}
   \bibitem{my.paper} M.\,Patrakeev, The complement of a $\hspace{1pt}\mathsurround=0pt\sigma$-compact subset of a space with a $\hspace{1pt}\mathsurround=0pt\pi$-tree also has a $\hspace{1pt}\mathsurround=0pt\pi$-tree, \emph{Topology and its Applications},  the special issue dedicated to the 120th  anniversary of P.S. Alexandroff, accepted for publication in 2016. Preprint: \verb"https://arxiv.org/abs/1512.02458"\:.

   \bibitem{MPatr} M.\,Patrakeev, Metrizable images of the Sorgenfrey line, \emph{Topology Proceedings} vol.\,45 (2015) 253--269.

	\bibitem{arh1} A.\,V. Arhangel'skii, \emph{Open and close-to-open mappings. Relations among spaces}, Trudy Moskov Mat. Ob\v{s}\v{c}. 1966, no.\,{15}, 181--223.

   \bibitem{kech} A.\,S.\,Kechris, \emph{Classical descriptive set theory}, Graduate Texts in Mathematics, vol.\,156, Springer, 1994.

   \bibitem{Juh} I.\,Juh\'{a}sz, \emph{Cardinal functions in topology: ten years later,} MC tracts, vol.\,123, 1980.

   \bibitem{Pacific} E.\,K.\,van Douwen, W.\,F.\,Pfeffer, Some properties of the Sorgenfrey line and related spaces, \emph{Pacific J. Math.}, vol.\,81, no.\,{2} (1979), 371--377.

   \bibitem{jech} T.\,Jech, \emph{Set Theory}, $2{\text{nd}}$ Edition, Springer, 1996.

   \bibitem{kun} K.\,Kunen, \emph{Set Theory}, North-Holland, 1980.

   \bibitem{top.enc} K.\,P.\,Hart, J.\,Nagata, and J.\,E.\,Vaughan, eds., \emph{Encyclopedia of general topology}, Elsevier, Amsterdam, 2004.

   \bibitem{patr.slides} M.\,Patrakeev, \emph{The foliage hybrid operation}, slides from the talk at Alexandroff Readings, May 22--26, 2016, Moscow, Russia; available at \verb"https://arxiv.org/src/1512.02458v5/anc"\:.

\end{thebibliography}
\end{document}